\documentclass[11pt]{amsart}
\usepackage{amssymb}
\usepackage{latexsym,amsfonts,graphicx,amsmath}
\usepackage{epstopdf}
\usepackage{enumerate}
\usepackage{enumitem}
\usepackage{booktabs}
\usepackage{stmaryrd}
\usepackage{sidecap}
\usepackage{color}
\usepackage{bm}
\usepackage[font=small,labelfont=bf]{caption}
\numberwithin{equation}{section}
\graphicspath{{Images/}}

\def\vv#1{\text{\rm\bfseries #1}}
\def\tr{\text{\rm tr}}
\def\div{\text{\rm div}}
\def\jump#1{\llbracket {#1} \rrbracket}

\numberwithin{equation}{section}

\usepackage{amsthm}
\newtheorem{theorem}{Theorem}[section]

\newtheorem{lemma}{Lemma}[section]

\theoremstyle{definition}

\newtheorem{assumption}{Assumption}
\newtheorem{corollary}{Corollary}[section]

\begin{document}

\title[Unfitted Method for Stokes]{New stability estimates for an unfitted  finite element method for two-phase Stokes problem}

\author[E. C\'aceres]{Ernesto C\'aceres\textsuperscript{\textdagger}}
\address{\textsuperscript{\textdagger} Division of Applied Mathematics, Brown University, Providence, RI 02912, USA}
\email{ernesto\_caceres\_valenzuela@brown.edu}

\author[J. Guzm\'an]{Johnny Guzm\'an\textsuperscript{\textdagger}}
\address{\textsuperscript{\textdagger} Division of Applied Mathematics, Brown University, Providence, RI 02912, USA}
\email{johnny\_guzman@brown.edu}
\thanks{Partially supported by NSF through the Division of Mathematical Sciences grant  1620100 }

\author[M.A. Olshanskii]{Maxim Olshanskii\textsuperscript{\textdaggerdbl}}
\address{\textsuperscript{\textdaggerdbl}Department of Mathematics, University of Houston, Houston, TX 77204, USA}
\email{molshan@math.uh.edu}
\email{}
\thanks{Partially supported by NSF through the Division of Mathematical Sciences grant 1717516.}

\maketitle

\date{}

\begin{abstract}
	The paper addresses stability and finite element analysis of the stationary two-phase Stokes problem with a piecewise constant viscosity coefficient experiencing a jump across the interface between two fluid phases. We first prove \textit{a priori} estimates for the individual terms of the Cauchy stress tensor with stability constants independent of  the  viscosity coefficient. Next, this stability result is extended to the approximation of the  two-phase Stokes problem by a finite element method. In the  method considered, the interface between the phases does not respect the underlying triangulation, which put the finite element method into the class of unfitted discretizations. The finite element error estimates are proved with constants independent of viscosity.
	Numerical experiments supporting the theoretical results are provided.
\end{abstract}

\maketitle
\noindent
%{\bf Key words:}

%================================================================================
%																	 INTRODUCTION
%================================================================================
\section{Introduction} We are interested in the analysis and a finite element method for the two-phase Stokes problem (also known in the literature as \emph{the Stokes  interface problem}). The system of equations is posed in a bounded  Lipschitz domain $\Omega\subset \mathbb R^d$, $d=2,3$, decomposed in two subdomains (phases) $\Omega^\pm$. The interface $\Gamma$ between two phases is a closed hypersurface immersed in $\Omega$, i.e., $\Gamma\subset\Omega$ and  $\Gamma=\overline{\Omega^+}\cap\overline{\Omega^-}$. We assume  $\Gamma$ is Lipschitz  smooth.
The Stokes interface problem reads as follows:  Given a force field $\vv f\in L^2(\Omega)^d$, a source term $g\in L^2(\Omega)$, an interface force  {$\bm{\lambda}\in L^2(\Gamma)^d$}, and viscosity coefficient  $\nu^\pm$
constant and positive in each subdomain,
find the fluid velocity $\vv{u}$ and the normalized kinematic pressure $p$ such that
\begin{equation}\label{eq:contform}
\begin{aligned}
-\div\big(\nu^\pm D(\vv u)\big)+\nabla p^\pm & =\vv f^\pm &~\text{in}&\quad \Omega^\pm,\\
\div\,\vv u & = g& ~\text{in}&\quad\Omega,\\
\jump{\vv u} & = \vv 0&~\text{on}&\quad\Gamma,\\
\jump{\sigma(\vv u,p)\,\vv n} & =  {\bm{\lambda}}& ~\text{on}&\quad\Gamma,\\
\vv u & = \vv 0&~\text{on}&~ \partial\Omega,
\end{aligned}
\end{equation}
where $D(\vv u):=\frac12(\nabla\vv u+(\nabla \vv u)^T)$ is the rate-of-strain tensor,  $\sigma(\vv u,p)=\nu\,D(\vv u)-p\,\mathbb{I}$ is the Cauchy stress tensor, and $\vv n$ is a unit vector on $\Gamma$ pointing from $\Omega^+$ to $\Omega^-$.
For any $f \in L^1(\Omega)$ we use notations  $f^\pm$ for the restriction of $f$ on $\Omega^\pm$, i.e., $f^\pm=f|_{\Omega^\pm}$; same convention is used for vector functions. The jumps on the interface are then defined as
$
\jump{\sigma(\vv u,p)\vv n}=\sigma(\vv u^+, p^+) \vv n - \sigma(\vv u^-, p^-)\vv n$  and $\jump{\vv u}=\vv u^+-\vv u^-$.

The studies of the Stokes interface problem are motivated by continuum  models of two-phase flows. If the fluid is treated as Newtonian incompressible with immiscible phases separated by the sharp interface, then the system~\eqref{eq:contform} is a reasonable model problem for the limit case of highly viscous fluid; see, e.g.,~\cite{Truesdell1960,drew1983mathematical,sussman1994level,sussman2007sharp}. It also appears as an auxiliary problem in numerical simulations of two-phase incompressible flows~\cite{gross2011numerical}.
According to the continuum surface force model, cf. \cite{brackbill1992continuum}, the effect of interfacial forces, such as the surface tension, are taken into account by using a  localized force term at the interface, i.e.,  {$\bm{\lambda}$} in \eqref{eq:contform}.

Problem \eqref{eq:contform} is linear and a standard weak formulation \eqref{eq:mixedform} renders it as a saddle-point problem, thus  yielding the well-posedness result  and leading to Galerkin numerical methods; see, e.g.,~\cite{girault2012finite,brezzi2012mixed}. This textbook analysis, however, does not provide an explicit information on the dependence of the stability and numerical errors estimates on the viscosity coefficient, in particular, on the ratio $\nu^+/\nu^-$, provided $\nu^-\le \nu^+$. This robustness question becomes important  if one addresses  numerical stability of Galerkin  methods, such as the finite element method, for the case of high variation in viscosity coefficient between two phases. The $\nu$-dependence of stability and finite element error estimates for \eqref{eq:contform} have been studied in the literature only recently; see \cite{or-2006,ohmori2009some,hansbo2014cut,kirchhart2016analysis}.
In those studies, stability and error analysis was done for the natural energy norm  of the problem. In particular, under certain further assumptions on $\Omega^\pm$,  the \textit{a priori} estimate from~\cite{or-2006} (proved there for $g=0$,  {$\bm{\lambda}=0$}) reads
\begin{equation}\label{EstEng}
\|\nu^{\frac12}D(\vv u)\|_{L^2(\Omega)}+\|\nu^{-\frac12} p\|_{L^2(\Omega)}\le C \|\nu^{-\frac12}\vv f\|_{L^2(\Omega)},
\end{equation}
with $C$ independent of $\nu$. Note that for \emph{single} phase Stokes problem, a simple scaling argument provides uniform estimates
for the quantities $\nu \,\vv u$ and $p$. Similar result does not follow from \eqref{EstEng} for the velocity and pressure in each of the phases. For example, for $\nu^+\vv u^+$ and $p^+$ the estimate \eqref{EstEng} yields
\begin{equation}\label{EstEng1}
\|\nu^+ D(\vv u^+)\|_{L^2(\Omega^+)}+\|p^+\|_{L^2(\Omega^+)}\le C (\|\vv f\|_{L^2(\Omega^+)}+\sqrt{\frac{\nu^+}{\nu^-}}\|\vv f\|_{L^2(\Omega^-)} ).
\end{equation}
We see that the right-hand side blows up for $\nu^- \to 0$. In the present paper, we prove the following stability  result for the solution of \eqref{eq:contform}:
\begin{equation}\label{EstNew}
\|\nu D(\vv u)\|_{L^2(\Omega)}+\|p\|_{L^2(\Omega)}\le C (\|\vv f\|_{L^2(\Omega)}+\| {\bm{\lambda}}\|_{L^2(\Gamma)}+\|\nu\,g\|_{L^2(\Omega)}),
\end{equation}
The improvement over \eqref{EstEng1} is clear: the re-scaled solution components, $\nu^\pm \vv u^\pm$ and $p^\pm$, enjoy uniform estimates in the corresponding subdomains, just as for the single phase problem. The estimate \eqref{EstNew} can be also seen as \emph{the uniform estimate for the components of Cauchy stress tensor}, an important quantity in practical fluid mechanics.

In the same spirit as  \eqref{EstNew} improves over the energy estimate  \eqref{EstEng}, the finite element analysis developed in this paper extends the existing one by
deriving robust in $\nu$ stability estimates and error estimates for the  components of the finite element Cauchy stress tensor.
Following~\cite{hansbo2014cut}, for the discretization of \eqref{eq:contform} we consider a geometrically unfitted finite element method known as Nitsche-XFEM or cutFEM. Geometrically unfitted methods use a fixed background mesh which does not respect the position of the interface. The main advantage of unfitted FEM is the relative ease of handling time-dependent domains, implicitly defined interfaces and problems with strong geometric deformations~\cite{bordas2018geometrically}.
We prove uniform with respect to $\nu$ stability and error estimates  for the unfitted FEM. These results hold for a family of bulk LBB-stable finite element Stokes pairs defined on the background mesh. These pairs include $P_{k+1}-P_k$, $k\ge 1$, and   $P_{k+d}-P_{k}^{\rm disc}$ for $k\ge0$, $\Omega\subset\mathbb{R}^d$, $d=2,3$, and several other elements. We are able to accomplish this by combining ideas from the two papers \cite{guzman-olshanskii,bgss-2017}. In \cite{guzman-olshanskii} an unfitted FEM for the single phase Stokes problem was analyzed. We borrow some crucial inf-sup stability estimates from that paper.
In \cite{bgss-2017} similar stability results were proved the Poisson interface problem. The chief tool was to use extension operators in Sobolev spaces. Similarly, here extension operators are essential, however, the pressure terms and the div-free condition add several new difficulties.

We organized the paper in five sections. In section~\ref{sec2} a notion of the weak solution is introduced and estimate  \eqref{EstEng1} is proved. Section~\ref{sec3} describes the finite element method and proves the analogue of \eqref{EstEng1} for the finite element solution. In section~\ref{sec4} a $\nu$-independent optimal order error estimate is proved. Finally, a few illustrative results of numerical experiments are given in section~\ref{sec5}.

\section{\emph{A priori} analysis for \eqref{eq:contform} }
\label{sec2}
\subsection{Preliminaries and problem setting}
We introduce a variational formulation of \eqref{eq:contform} and several notations to be used throughout the paper. For an open set $\mathcal O\subset\mathbb R^d$  denote by $(\cdot,\cdot)_{\mathcal O}$ the $L^2$ inner product in $\mathcal O$, and by  $\|\cdot\|_{\mathcal O}$ the  corresponding norm.  For the mixed variational formulation of \eqref{eq:contform}, we set $\vv V:=[\mathrm{H}_0^1(\Omega)]^d$ for the space the vector field $\vv u$ belongs to, whereas for the pressure $p$ we set $M=L_0^2(\Omega)$, with $L_0^2(\mathcal O )=\{p\in L^2(\mathcal O):(p,1)_{\mathcal O}=0\}$. We let $\| \cdot \|_{1,\mathcal O}$ denote the $H^1({\mathcal O})$-norm.
 {
 The norm of $\vv V^*$, the dual of $\vv V$, is denoted by $\|\cdot\|_{-1}$, and
$\langle \cdot,\cdot\rangle_{-1}$ denotes the  pairing, with respect to the $L^2$-duality.
}

We consider the abstract mixed formulation: Find $(\vv u,p)\in \vv V\times M$ such that
\begin{equation}\label{eq:mixedform}
\begin{split}
a(\vv u,\vv v) +  b(\vv v,p)  & =  \langle \widehat{\vv f},\vv v\rangle_{-1} \qquad   \forall\,\vv v\in \vv V,\\
b(\vv u,q) & =  -(g, q)_{\Omega} \qquad \forall\,q\in M,
\end{split}
\end{equation}
where
\[a(\vv u,\vv v):=(\nu \,D(\vv u),D(\vv v))_{\Omega},\quad b(\vv v,q):=-(\div\,\vv v,q)_{\Omega},\quad \text{and}~~\widehat{\vv f}\in\vv V^*.
\] %{and where $\widehat{\vv f}\in\vv V^*$}.

The problem \eqref{eq:mixedform} is the weak formulation of the Stokes interface equation \eqref{eq:contform}  if we let
\begin{equation}\label{eq:Functional}
\langle \widehat{\vv f},\vv v\rangle_{-1}:=\int_\Omega\vv f\cdot \vv v+\int_\Gamma\bm{\lambda}\cdot\vv v.
\end{equation}

\subsection{Stability estimates for the weak solution}
%\label{sec:wpcont}
In this section, we analyze the variational formulation   {\eqref{eq:mixedform} of the Stokes interface problem} \eqref{eq:contform}.
We are interested in the following stability result for the solution $(\vv u,p)\in \vv V\times M$ of \eqref{eq:mixedform}:
\begin{equation}\label{eq:contaim}
\|\nu D(\vv u)\|_{\Omega}+\|p\|_{\Omega}\le C(\| {\widehat{\vv f}}\|_{-1}+ \|\nu g\|_{\Omega}),
\end{equation}
with a constant $C>0$ independent of $\nu$ and depending only on $\Omega$ and $\Gamma$.  A standard energy argument gives estimates for $\|\sqrt{\nu}\, D(\vv u)\|_{\Omega}$ with the corresponding constant $C$ on the right-hand side dependent on $\nu$.
Without any loss of generality we shall always assume $\nu^{-}\le\nu^{+}$. The key to get the improved result \eqref{eq:contaim} is using an energy argument to estimate $\|\nu^- D(\vv u)\|_{\Omega^-}$,  {and} then using an extension operator to estimate $\|\nu^+ D(\vv u)\|_{\Omega^+}$. This strategy is similar to  {the one taken in \cite{bgss-2017} for the Poisson interface problem}. However, here the pressure term and the divergence condition require careful treatment, by a repetitive use of a continuous inf-sup condition.

In what follows, for an open set $\mathcal O$ and a function $q\in L^2(\mathcal O)$, we denote its average over $\mathcal O$ by the expression $\mathrm{avg}_{\mathcal O}(q)=|\mathcal O|^{-1}(q,1)_{\mathcal O}$. We start by proving the following results.

\begin{lemma}\label{lemma:ppm}
 Let $\vv u \in \vv V$ and $p \in M$ solve \eqref{eq:mixedform}. Then there exists $C>0$, depending only on $\Omega$ and $\Gamma$, such that
 \begin{equation}\label{eq:ppm}
 \|p^{\pm} -\mathrm{avg}_{ \Omega^\pm}(p^\pm)\|_{\Omega^\pm}\le C ( \|\nu^\pm D(\vv u^\pm)\|_{\Omega^\pm}+\| {\widehat{\vv f}}\|_{-1}),
 \end{equation}
\begin{equation}\label{eq:pestimate}
\|p\|_{\Omega}\le C (\|\nu D(\vv u)\|_{\Omega}+\| {\widehat{\vv f}}\|_{-1}).
\end{equation}
\end{lemma}
\begin{proof}
Since $p^\pm-\mathrm{avg}_{\Omega^\pm}(p^\pm)\in L_0^2(\Omega^\pm)$, the result from \cite{bogovskii1979solution} ensures the existence of $\vv v^\pm\in \vv H_0^1(\Omega^\pm)$ such that $\div\,\vv v^\pm=p^\pm-\mathrm{avg}_{\Omega^\pm}(p^\pm)$ in $\Omega^\pm$, and $\|\vv v^\pm\|_{1,\Omega^\pm}\le C\|p^\pm-\mathrm{avg}_{\Omega^\pm}(p^\pm)\|_{\Omega}$, for some $C>0$ depending only on $\Omega$ and $\Gamma$. Therefore, using the first equation of \eqref{eq:mixedform} with $\vv v^\pm$ in $\Omega^\pm$  {extended by zero} on $\Omega^\mp$ and since $(\div\,\vv v^\pm,1)_{\Omega^\pm}=0$, we get
\[
\begin{split}
&\|p^\pm-\mathrm{avg}_{\Omega^\pm}(p^\pm)\|_{\Omega^\pm}^2=(p^\pm-\mathrm{avg}_{\Omega^\pm}(p^\pm),\div\,\vv v^\pm)_{\Omega^\pm}=(\div\,\vv v^\pm,p^\pm)_\Omega\\
&\qquad=\big(\nu^\pm D(\vv u^\pm),D(\vv v^\pm)\big)_{\Omega^\pm}-\langle  {\widehat{\vv f}}, {\vv v^{\pm}}\rangle_{-1} \le \big(\|\nu^\pm D(\vv u^\pm)\|_{\Omega^\pm}+\| {\widehat{\vv f}}\|_{-1}\big) \|\vv v^\pm\|_{1, \Omega^\pm} \\
&\qquad\le C\big(\|\nu^\pm D(\vv u^\pm)\|_{\Omega^\pm}+\| {\widehat{\vv f}}\|_{-1}\big)\|p^\pm-\mathrm{avg}_{\Omega^\pm}(p^\pm)\|_{\Omega^\pm},
\end{split}
\]
which gives \eqref{eq:ppm}. On the other hand, since $p\in M$, there exists $\vv v\in \vv V$ such that $\div\,\vv v=p$ in $\Omega$, and $\|\vv v\|_{1,\Omega}\le C\|p\|_{\Omega}$, for some $C>0$, depending only on $\Omega^\pm$. Therefore, using the first equation of \eqref{eq:mixedform} with $\vv v\in \vv V$, we get \[\|p\|_{\Omega}^2=(\div\,\vv v,p)_{\Omega}=(\nu D(\vv u),D(\vv v))_{\Omega}-\langle {\widehat{\vv f}},\vv v\rangle_{-1}\le C\|p\|_{\Omega}\big(\|\nu D(\vv u)\|_{\Omega}+\| {\widehat{\vv f}}\|_{-1}\big), \]which proves \eqref{eq:pestimate}.
\end{proof}

In order to prove \eqref{eq:contaim}, we start by estimating $\nu^-D(\vv u^{-})$. This is an easier part of the desired  estimate \eqref{eq:contaim} since $\nu^{-}\le\nu^{+}$.
\begin{lemma}\label{lem:rho-}
Let $\vv u \in \vv V$ and $p \in M$ solve \eqref{eq:mixedform}, then there exists $C>0$, depending only on $\Omega$, such that
  \begin{equation}\label{eq:errrhom}
  \|\nu^-D(\vv u^-)\|_{\Omega^-}\le C(\| {\widehat{\vv f}}\|_{-1}+ \|\nu g\|_{\Omega}).
  \end{equation}
\end{lemma}
\begin{proof} The result follows from basic energy arguments and pressure estimates provided by Lemma~\ref{lemma:ppm}.
  To see this, let us test the first equation of \eqref{eq:mixedform} with $\vv v=\vv u$ and the second equation with $q=p$ to get
\begin{equation}\label{aux1000}
\|\nu^{\frac12} D(\vv u)\|_{\Omega}^2 =\langle {\widehat{\vv f}},\vv u\rangle_{-1} + (g, p)_{\Omega}.
\end{equation}
We start by estimating the first term on the right-hand side
\begin{equation}\label{aux1001}
 \langle {\widehat{\vv f}},\vv u\rangle_{-1} \le   C\| {\widehat{\vv f}}\|_{-1}\|D(\vv u)\|_{\Omega} \le \frac{C\| {\widehat{\vv f}}\|_{-1}}{\sqrt{\nu^-}}\|\nu^{\frac12} D(\vv u)\|_{\Omega}
 \end{equation}
 where we used  Poincare's inequality, Korn's inequality and the fact that $0<\nu^-\le \nu^+$. To estimate the second term, we use the following decomposition
 \[
 \begin{split}
 (g, p)_{\Omega}=&  (g^-, p^--\mathrm{avg}_{\Omega^-}(p^-))_{\Omega^-} +( g^-, \mathrm{avg}_{\Omega^-}(p^-))_{\Omega^-} \\
 &\,+(g^+, p^+-\mathrm{avg}_{\Omega^+}(p^+))_{\Omega^+} +( g^+, \mathrm{avg}_{\Omega^+}(p^+))_{\Omega^+}.
 \end{split}
 \]
 By \eqref{eq:ppm} we have
 \[
 \begin{split}
 & (g^-, p^--\mathrm{avg}_{\Omega^-}(p^-))_{\Omega^-}+ (g^+, p^+-\mathrm{avg}_{\Omega^+}(p^+))_{\Omega^+} \\
 & \qquad \le C ( \sqrt{\nu^-} \| g^-\|_{\Omega^-} \| \sqrt{\nu^-} D(\vv u^-)\|_{\Omega^-}+  \sqrt{\nu^+} \| g^+\|_{\Omega^+} \|\sqrt{\nu^+} D(\vv u^+)\|_{\Omega^+}) + C \|g\|_{\Omega} \| {\widehat{\vv f}}\|_{-1} \\
 & \qquad \le C ( \|\nu^{\frac12} g\|_{\Omega} \|\nu^{\frac12} D(\vv u)\|_{\Omega}+ \|g\|_{\Omega} \| {\widehat{\vv f}}\|_{-1}).
 \end{split}
 \]
From $g \in M$ it follows  that $(g,1)_{\Omega}=1$ and so $|\Omega^{+}|\mathrm{avg}_{\Omega^+}(p^+)=-|\Omega^{-}|\mathrm{avg}_{\Omega^-}(p^-)$.  We employ this equality below to obtain
 \[
\begin{split}
& ( g^-, \mathrm{avg}_{\Omega^-}(p^-))_{\Omega^-}+ ( g^+, \mathrm{avg}_{\Omega^+}(p^+))_{\Omega^+} \\
&\qquad\qquad= ( \mathrm{avg}_{\Omega^-}(g^-), \mathrm{avg}_{\Omega^-}(p^-))_{\Omega^-}+ (  \mathrm{avg}_{\Omega^+}(g^+), \mathrm{avg}_{\Omega^+}(p^+))_{\Omega^+} \\
&\qquad\qquad= -\frac{|\Omega^{+}|}{|\Omega^{-}|} ( \mathrm{avg}_{\Omega^+}(g^+), \mathrm{avg}_{\Omega^-}(p^-))_{\Omega^-} + (  \mathrm{avg}_{\Omega^+}(g^+), \mathrm{avg}_{\Omega^+}(p^+))_{\Omega^+}.
\end{split}
\]
Hence, we have
\[
\begin{split}
&( g^-, \mathrm{avg}_{\Omega^-}(p^-))_{\Omega^-}+ ( g^+, \mathrm{avg}_{\Omega^+}(p^+))_{\Omega^+}\le C\, \|g^+\|_{\Omega^+} \|p\|_{\Omega} \le C\, \|g^+\|_{\Omega^+} (\|\nu D(\vv u)\|_{\Omega}+ \| {\widehat{\vv f}}\|_{-1}) \\
& \qquad\le C\, \sqrt{\nu^+} \|g^+\|_{\Omega^+} \|\nu^{\frac12} D(\vv u)\|_{\Omega} + C\|g\|_{\Omega} \| {\widehat{\vv f}}\|_{-1},
\end{split}
\]
where we used \eqref{eq:pestimate} and the fact that $\nu^- \le \nu^+$.
Thus, we have shown that
\begin{equation}\label{aux1002}
%\begin{split}
(g, p)_{\Omega} \le C ( \|\nu^{\frac12} g\|_{\Omega} \|\nu^{\frac12} D(\vv u)\|_{\Omega}+ \|g\|_{\Omega} \| {\widehat{\vv f}}\|_{-1}).
%\end{split}
\end{equation}
Combining \eqref{aux1000},  \eqref{aux1001},  \eqref{aux1002}  and using that  $\nu^- \le \nu^+$ we arrive at
\begin{equation*}
\|\nu^{\frac12} D(\vv u)\|_{\Omega} \le C \bigg(\frac{1}{\sqrt{\nu^-}}\| {\widehat{\vv f}}\|_{-1}+\|\nu^{\frac12} g\|_{\Omega} \bigg).
\end{equation*}
This implies the result due to $\nu^- \le \nu^+$.
\end{proof}

An immediate consequence of Lemma \ref{lem:rho-} and \eqref{eq:ppm} is the desired pressure estimate in $\Omega^{-}$.
\begin{lemma} \label{lem:boundpavg}
 Let $\vv u \in \vv V$ and $p \in M$ solve \eqref{eq:mixedform},  then there exists $C>0$, depending only on $\Omega$, such that
 %\begin{equation}\label{eq:errp-p}
 \[
  \|p^--\mathrm{avg}_{\Omega^-}(p^-)\|_{\Omega^-}\le C(\| {\widehat{\vv f}}\|_{-1}+ \|\nu g\|_{\Omega}).
 \]
  %\end{equation}
\end{lemma}

It remains to show analogues bound for $\nu D(\vv u)$ and $p$ on $\Omega^+$.
To estimate $\nu D(\vv u)$ in $\Omega^+$, we consider an extension operator. A detailed construction of this operator can be found in \cite[Chapter VI, Section 3.3]{SteinBook}. Let $E:\vv H^1(\Omega^+)\to \vv H^1(\Omega)$ be the bounded extension operator from $\Omega^+$ to $\Omega$. That is, if $\vv w\in \mathbf H^1(\Omega^+)$, then
\begin{equation}\label{eq:extprop}
E\vv w\in \mathbf H^1(\Omega),\quad E\vv w=\vv w\,\,\,\text{ in }\,\,\,\Omega^+,\quad\text{and}\quad \|E\vv w\|_{1,\Omega}\le C\|\vv w\|_{1,\Omega^+},
\end{equation}
with a  constant $C>0$, depending only on $\Omega^+$ and $\Omega$. Moreover, one can assume $E\vv w$ to vanish on $\partial \Omega$, i.e., $E \vv w \in \vv V$ for $\vv w \in\vv V$. Further, we note that if $\vv w(\vv x):=\vv a+\vv B\vv x\in \mathbb{RM}(\Omega)$, then
%\begin{equation}\label{eq:proprm}
\[
D(\vv w)=\vv 0,\quad  {\text{and hence}~ \div \,\vv w=\tr(D(\vv w))=0}.
\]
%\end{equation}
Here $\mathbb{RM}(\mathcal O)$ is the space of  {rigid} body motions defined on $\mathcal O$. We denote by $\mathcal P^{RM}_{\mathcal O}( \vv v)$ the $L^2$-orthogonal projection of $\vv v\in L^2(\mathcal O)^d$ onto the subspace of rigid body motions.
%, i.e.
%\begin{equation*}
%$
%(\mathcal P^{RM}_{\mathcal O}( \vv v) , \vv w)_{\mathcal O} =  ( \vv v , \vv w)_{\mathcal O}$ for all $\vv w \in  \mathbb{RM}(\mathcal O).
%$
%\end{equation*}

\begin{lemma}\label{lemma:eq:errorflux}
Let $\vv u \in \vv V$ and $p \in M$ solve \eqref{eq:mixedform},  then there exists $C>0$, depending only on $\Omega$ and $\Gamma$, such that
\begin{equation}\label{eq:errorflux}
\|\nu^+D(\vv u^+)\|_{\Omega^+}\le
C\big(\| {\widehat{\vv f}}\|_{-1} +\|\nu g\|_{\Omega}\big).
\end{equation}
\end{lemma}
\begin{proof}
Define $\vv v\in\vv V$ by
%\begin{equation}\label{eq:defext}
\[
\vv v=\begin{cases}E\vv u^+,&\text{if }|\partial\Omega^+\cap\partial\Omega|>0,\\ E(\vv u^+-\mathcal P^{RM}_{\Omega^+}(\vv u^+)),&\text{if }|\partial\Omega^+\cap\partial\Omega|=0.
\end{cases}
\]
%\end{equation}
The boundeness of the extension operator in \eqref{eq:extprop}, Poincare's and Korn's inequalities imply
\begin{equation} \label{aux111}
\|\vv v\|_{1, \Omega} \le C \, \|D(\vv u^+)\|_{\Omega^+}.
\end{equation}
The first equation of \eqref{eq:mixedform} with the above $\vv v$ gives
\begin{equation}\label{aux2001}
  \nu^+\|D(\vv u^+)\|_{\Omega^+}^2  =\langle {\widehat{\vv f}},\vv v\rangle_{-1}-(\nu^-D(\vv u^-),D(\vv v^-))_{\Omega^-}+(\div\, \vv v,p)_{\Omega},
\end{equation}
where we used that $ D(\vv v^+)= D(\vv u^+)$.
We now bound each term on the right-hand side of \eqref{aux2001}. Using \eqref{aux111} we have
\begin{equation}\label{aux2002}
\langle {\widehat{\vv f}},\vv v\rangle_{-1} \le C \, \| {\widehat{\vv f}}\|_{-1} \|D(\vv u^+)\|_{\Omega^+}.
\end{equation}
Thanks to \eqref{aux111}  and \eqref{eq:errrhom}  we bound the second term on the right hand side of \eqref{aux2001}
\begin{equation}\label{aux2003}
-(\nu^-D(\vv u^-),D(\vv v^-))_{\Omega^-} \le C \, (\| {\widehat{\vv f}}\|_{-1}  +\|\nu g\|_{\Omega})   \|D(\vv u^+)\|_{\Omega^+}.
\end{equation}
For the third term we use decomposition:
\begin{equation}\label{aux324}
(\div\, \vv v,p)_{\Omega}= (\div\, \vv v^-,p^--\mathrm{avg}_{\Omega^-}(p^-))_{\Omega^-}+(\div\, \vv v^-,\mathrm{avg}_{\Omega^-}(p^-))_{\Omega^-}+(\div\, \vv u^+, p^+)_{\Omega^+}
\end{equation}
Using \eqref{lem:boundpavg}  and \eqref{aux111} we estimate
\begin{equation*}
(\div\, \vv v^-,p^--\mathrm{avg}_{\Omega^-}(p^-))_{\Omega^-} \le C \, (\| {\widehat{\vv f}}\|_{-1}  +\|\nu g\|_{\Omega})   \|D(\vv u^+)\|_{\Omega^+}.
\end{equation*}
For the next term in \eqref{aux324} we use $(\div \,  \vv v, 1)_{\Omega}=0$ and the second equation in \eqref{eq:contform} to write
\[
\begin{split}
&(\div\, \vv v^-,\mathrm{avg}_{\Omega^-}(p^-))_{\Omega^-} \\
&\qquad =(\mathrm{avg}_{\Omega^-}(\div\, \vv v^-), \mathrm{avg}_{\Omega^-}(p^-))_{\Omega^-}
=-\frac{|\Omega^+|}{|\Omega^-|} (\mathrm{avg}_{\Omega^+}(\div\, \vv v^+), \mathrm{avg}_{\Omega^-}(p^-))_{\Omega^-} \\
&\qquad=-\frac{|\Omega^+|}{|\Omega^-|} (\mathrm{avg}_{\Omega^+}(\div \, \vv u^+), \mathrm{avg}_{\Omega^-}(p^-))_{\Omega^-} =-\frac{|\Omega^+|}{|\Omega^-|} (\mathrm{avg}_{\Omega^+}(g^+), \mathrm{avg}_{\Omega^-}(p^-))_{\Omega^-}.
\end{split}
\]
Hence, with the help of  \eqref{eq:pestimate} and \eqref{eq:errrhom} we get
\[
\begin{split}
(\div\, \vv v^-,\mathrm{avg}_{\Omega^-}(p^-))_{\Omega^-} & \le C \, \|g^+\|_{\Omega^+} \|p\|_{\Omega} \le C \|g^+\|_{\Omega^+} (\|\nu D(\vv u)\|_{\Omega} +\|\widehat{\vv f}\|_{-1})  \\
& \le C \nu^+ \|g^+\|_{\Omega^+}  \|D(\vv u^+)\|_{\Omega^+}+  C \|g^+\|_{\Omega^+} (\| {\widehat{\vv f}}\|_{-1}+\|\nu g\|_{\Omega}).
\end{split}
\]
Similarly, for the last term in \eqref{aux324} we have,
\[
%\begin{split}
(\div\, \vv u^+, p^+)_{\Omega^+} \le C \nu^+ \|g^+\|_{\Omega^+}  \|D(\vv u^+)\|_{\Omega^+} +  C \|g^+\|_{\Omega^+} (\| {\widehat{\vv f}}\|_{-1}+\|\nu g\|_{\Omega}).
%\end{split}
\]
We combine the last two estimates to obtain the bound
\[
%\begin{split}
(\div\, \vv v,p)_{\Omega}  \le (\| {\widehat{\vv f}}\|_{-1}  +\|\nu g\|_{\Omega})   \|D(\vv u^+)\|_{\Omega^+} + C \|g^+\|_{\Omega^+} (\| {\widehat{\vv f}}\|_{-1}+\|\nu g\|_{\Omega}).
%\end{split}
\]
This estimate together with \eqref{aux2003}, \eqref{aux2002}, \eqref{aux2001}  leads to another bound
\[
\begin{split}
  \nu^+\|D(\vv u^+)\|_{\Omega^+}^2 & \le (\| {\widehat{\vv f}}\|_{-1}  +\|\nu g\|_{\Omega})   \|D(\vv u^+)\|_{\Omega^+} + C \|g^+\|_{\Omega^+} (\| {\widehat{\vv f}}\|_{-1}+\|\nu g\|_{\Omega})\\
&  {\le \frac{1}{2\nu^+}(\|\widehat{\vv f}\|_{-1}  +\|\nu g\|_{\Omega})^2   +\frac{\nu^+}{2}\|D(\vv u^+)\|_{\Omega^+}^2 }+ C \|g^+\|_{\Omega^+} (\| {\widehat{\vv f}}\|_{-1}+\|\nu g\|_{\Omega}).
\end{split}
\]
This implies \eqref{eq:errorflux}  {after multiplication all through by $\nu^+$ and doing simple computations}.
\end{proof}
Collecting  \eqref{eq:pestimate}, \eqref{eq:errrhom} and \eqref{eq:errorflux} we obtain the main result of this section.
\begin{theorem}
	%\label{thm:cont}
 Let $\vv u \in \vv V$ and $p \in M$ solve \eqref{eq:mixedform}, then there exists $C>0$, depending only on $\Omega$ and $\Gamma$, such that
 \begin{equation} \label{eq:estflux}
  \|\nu D(\vv u)\|_{\Omega}+\|p\|_{\Omega}\le C(\| {\widehat{\vv f}}\|_{-1}  +\|\nu g\|_{\Omega})\le
  C (\|\vv f\|_{\Omega}+\| {\bm{\lambda}}\|_{\Gamma}+\|\nu\,g\|_{\Omega}).
 \end{equation}
\end{theorem}

The second inequality in \eqref{eq:estflux} follows from the definition of the functional $\widehat{\vv f}$ in \eqref{eq:Functional}, the Poincar\'e inequality and the trace inequality, $\|\vv v\|_{\Gamma}\le C \|\vv v\|_{1,\Omega}$.

\section{The finite element method}
\label{sec3}
\subsection{Preliminaries and problem setting}
\label{sec:prelimdisc}
For the discretization purpose, we assume that $\Omega$ is polygonal/polyhedral.
Let $\{\mathcal T_h\}_{h>0}$ be an admissible family of triangulations of $\Omega$. We adopt the convention that the elements $T$ and edges $e$ are open sets, and use over-line symbol to refer to their closure. For each simplex $T\in\mathcal T_h$, let $h_T$ denote its diameter and define the global parameter of the triangulation by $h=\max_{T}h_T$. We assume that $\mathcal T_h$ is shape regular, that is, there exists $\kappa>0$ such that for every $T\in\mathcal T_h$, the radius $\rho_T$ of the inscribed sphere satisfies
\begin{equation}\label{eq:minang}
\rho_T>\frac{h_T}{\kappa}.
\end{equation}
The  {sets} of elements intersecting $\Omega^\pm$ and the set of elements cutting the interface $\Gamma$ are of interest.
These are defined by
\[\mathcal T_h^\pm:=\{T\in\mathcal T_h:T\cap\Omega^\pm\neq\emptyset\},\quad\text{and}\quad\mathcal T_h^\Gamma:=\{T\in\mathcal {T}_h:\overline T\cap \Gamma\neq\emptyset\}.
\]
We also  {define the sets of elements interior to each of subdomains $\Omega^\pm$},   $\mathcal T_{h, i}^\pm=\{ T \in \mathcal T_h^\pm: T \subset \Omega^\pm\}$.  Finally, we let $\mathcal{E}_{h,i}^{\pm}$ be the collection of $d-1$, sub-simplexes of  $\mathcal T_{h, i}^\pm$ {(faces for $d=3$ and edges for $d=2$)}.

For $T\in\mathcal T_h^\Gamma$, we denote $T_\Gamma:=\overline T\cap \Gamma$. Under these definitions we define the $h$-dependent domains
\[
\Omega_h^\pm:=\mathrm{int}\Big(\bigcup_{T\in\mathcal T_h^\pm}\overline T\Big),\quad \text{and}\quad\Omega_{h,i}^\pm:=\mathrm{int}\Big(\bigcup_{T\in\mathcal T_{h,i}^\pm} \overline T\Big).
\]
In particular, using the definition of the sets $\mathcal T_h^+$ and $\mathcal T_{h,i}^-$, we have that $\overline{\Omega}=\overline{\Omega_{h,i}^-}\cup\overline{\Omega_h^+}$. This fact will be useful when constructing a discrete extension operator. We {also consider the layer of elements cut by the interface:}
\[
\omega_h:=\mathrm{int}\Big(\bigcup_{T\in\mathcal T_h^\Gamma}\overline T\Big), %\quad\text{and}~ \Omega_h:=\Omega\backslash \overline{\omega_h},
\]
and define the set of  {faces (edges)} of $\mathcal T_h^\Gamma$ restricted to the interior of $\Omega_h^\pm$:
\[\mathcal E_h^{\Gamma,\pm}:=\{e=\mathrm{int}(\partial T_1\cap\partial T_2):T_1,T_2\in\mathcal T_h^\pm\text{ and }T_1\cap\Gamma\neq\emptyset\text{ or }T_2\cap\Gamma\neq\emptyset\}.\]
For a piecewise smooth vector valued function $\vv v$, the jump across an interior  {face} $e=\mathrm{int}(\partial T_1\cap\partial T_2)$ is defined by
$\jump{\vv v }=\vv v |_{T_1}\cdot \vv n_1+\vv v|_{T_2}\cdot \vv n_2$,
 where $\vv n_1$ and $\vv n_2$ are the unit normal vectors to $e$, pointing outwards to $T_1$ and $T_2$, respectively.  For a scalar function, we define
 $\jump{p }=p |_{T_1}\vv n_1+p|_{T_2} \vv n_2$.

The space of discontinuous and continuous  {finite element} pressures are given by
\[
M_{h,\rm disc}:=\{q \in L^2(\Omega): q|_T\in\mathrm P_{k_p}(T)\quad\forall\,T\in\mathcal T_h\}\quad\text{and}\quad M_{h,\rm cont}:=M_{h,\rm disc}\cap C(\Omega),
\]
 {where integer $k_p\ge0$ is a fixed polynomial degree.}
Throughout this paper, $ {M_h^{\rm bulk}}= M_{h, \rm disc}$ for $k_p \ge 0$ or $ {M_h^{\rm bulk}}= M_{h, \rm cont}$ for  $k_p \ge 1$. We define $M_h^\pm:= {M_h^{\rm bulk}} \cap L^2(\Omega_h^\pm)$.  {Finally,} our pressure space is given by
\[
M_h:=\big\{q:=(q^-,q^+)\in M_h^- \times M_h^+ : (q^-,1)_{\Omega^-}+(q^+,1)_{\Omega^+}=0\big\}.
\]
Note that every element from $\mathcal T_h^\Gamma$ supports two finite element pressures corresponding to different phases.
Only the restriction of these pressures to $\Omega^+$ or $\Omega^+$, respectively, makes sense as a numerical approximation of the true pressure solving the original problem \eqref{eq:contform}. Same comment will be valid for the finite element velocity fields defined next.
We consider  the vector finite element space for $k \ge 1$,
\[
\vv W_h^k=\{ \vv w \in \vv V: \vv w \in \vv P_k(T), \text{ for all } T \in \mathcal{T}_h\}.
\]
Next we consider a background velocity finite element space $ {\vv V_h^{\rm bulk}}$ such that
\[
\vv W_h^{k_u}\subset {\vv V_h^{\rm bulk}} \subset \vv W_h^s,
\]
for some integers $s\ge k_u\ge 1$. Let $\vv V_h^\pm :=  {\vv V_h^{\rm bulk}} \cap \vv H^1(\Omega_h^\pm)$. Finally our velocity space will be
\[
\vv V_h:=\big\{\vv v:=(\vv v ^-,\vv v^+)\in \vv V_h^-\times \vv V_h^+  \big\}.
\]
Also, we denote a generic element $\vv v_h\in \vv V_h$ by $\vv v_h:=(\vv v_h^-,\vv v_h^+)$.

Functions from $M_h$, $\vv V_h$ and their derivatives are multivalued in $\omega_h$, the overlap of $\Omega_h^+$ and $\Omega_h^-$. Below we use the $L^2(\Omega)$ norm notion for such  functions to denote the norm of single-valued functions obtained by restricting $\Omega_h^+$-components on $\Omega^+\subset\Omega_h^+$ and $\Omega_h^-$-components on $\Omega^-\subset\Omega_h^-$. For example,
\[
\|p_h\|^2_\Omega=\|p_h^+\|^2_{\Omega^+}+\|p_h^-\|^2_{\Omega^-}\quad\text{or}\quad \|\nu \,D(\vv u_h)\|^2_{\Omega}=
 \|\nu^+ \,D(\vv u_h^+)\|^2_{\Omega^+}+ \|\nu^- \,D(\vv u_h^-)\|^2_{\Omega^-},
\]
for $p_h\in M_h$, $\vv u_h\in \vv V_h$, and so forth.  The jump of a multivalued function over the interface is defined as the difference of components coming from $\Omega_h^+$ and $\Omega_h^-$, i.e. $\jump{\vv v_h}=\vv v_h^+-\vv v_h^-$ on $\Gamma$.

 {We} consider a discrete norm $\|\cdot\|_{1,h}$, defined on $\vv V_h$, as follows: for all $\vv v_h\in\vv V_h$,
\[
\|\vv v_h\|_{1,h}^2:=\|D(\vv v_h)\|_{\Omega}^2+\sum_{T\in\mathcal T_h^\Gamma}\frac{\|\jump{\vv v_h}\|_{T_\Gamma}^2}{h_T}+\sum_{e\in \mathcal E_{h}^{\Gamma, -}}\sum_{\ell=1}^{s}|e|^{2\ell-1}\|\jump{\partial_{\vv n}^\ell\vv v_h^-}\|_{e}^2 + \sum_{e\in \mathcal E_{h}^{\Gamma, +}}\sum_{\ell=1}^{s}|e|^{2\ell-1}\|\jump{\partial_{\vv n}^\ell\vv v_h^+}\|_{e}^2,
\]
%\end{equation}
where $\partial_{\vv n}^\ell$ is the $\ell-$th order normal derivative.
 We define a scaled norm
\begin{alignat*}{1}
\|\vv v_h\|_{1,h,\nu}^2:=&\|\nu D(\vv v_h)\|_{\Omega}^2+\sum_{T\in\mathcal T_h^\Gamma}\frac{\|\nu^-\jump{\vv v_h}\|_{T_\Gamma}^2}{h_T}\\
&+\sum_{e\in \mathcal E_{h}^{\Gamma, -}}\sum_{\ell=1}^{s}|e|^{2\ell-1}\|\nu^- \jump{\partial_{\vv n}^\ell\vv v_h^-}\|_{e}^2 + \sum_{e\in \mathcal E_{h}^{\Gamma, +}}\sum_{\ell=1}^{s}|e|^{2\ell-1}\|\nu^+ \jump{\partial_{\vv n}^\ell\vv v_h^+}\|_{e}^2,
\end{alignat*}
 {and the augmented scaled norm
\begin{alignat}{1}
\|\vv v_h\|_{1,h,\nu,\star}^2:=&\|\vv v_h\|_{1,h, \nu}^2+ \sum_{T\in\mathcal T_h^\Gamma} h_T \|\nu^- D(\vv v_h^-)\|_{T_\Gamma}^2.\label{augmentednorm}
\end{alignat}
}
We use the notation $\|\cdot\|_{\pm,h}$ to define a discrete norm  on $M_h$, for all $q_h=(q_h^-, q_h^+) \in M_h$, by
\[
\|q_h^\pm\|_{\pm,h}^2:=\|q_h^\pm\|_{\Omega^\pm}^2+\sum_{e\in\mathcal E_h^{\Gamma,\pm}} \sum_{\ell=0}^{k_p}|e|^{2\ell+1}\|\jump{\partial_{\vv n}^\ell q_h^\pm}\|^2_{e},\quad  {\|q_h\|_{h}^2=\|q_h^+\|_{+, h}^2+ \|q_h^-\|_{-, h}^2.}
\]

Spaces ${\vv V_h^*}$ and $M_h^*$ are dual to $\vv V_h$ and $M_h^\pm$ with respect to $\|\cdot\|_{1,h}$ and $\|\cdot\|_{\pm,h}$. For $\vv F_h\in  {\vv V_h^*}$ and $G_h\in M_h^*$, we have by definition that
\[
\|\vv F_h\|_{-1,h}=\sup_{\vv v_h\in \vv V_h}\frac{\vv F_h(\vv v_h)}{\|\vv v_h\|_{1,h}},\quad\text{and}\quad \|G_h\|_{-1,h}=\sup_{q_h\in M_h}\frac{G_h(q_h)}{\|q_h\|_{h}},
\]
where we agree that $\sup_x$ is taken over non-zero elements, if $x$ appears in the denominator.
In this paper, we only consider  bulk spaces $ {\vv V_h^{\rm bulk}},  {M_h^{\rm bulk}}$ which  {form inf-sup stable pairs}.  This is listed as an assumption.

\begin{assumption}\label{Ass0}
 There exists a constant $\xi>0$ such that
\[
\xi \|q\|_{\Omega} \le \sup_{\vv v  \in  {\vv V_h^{\rm bulk}}} \frac{(\div \, \vv v, q)_{\Omega}}{\|\vv v\|_{1, \Omega}} \quad \text{ for all } q \in  {M_h^{\rm bulk}} \cap L_0^2(\Omega).
\]
\end{assumption}

We end this section by considering further assumptions on the mesh and on the pair of spaces $\{\vv V_h,M_h\}$. These assumptions are essentially the ones made in  \cite{guzman-olshanskii}. For a generic set of elements $\mathcal T\subset\mathcal T_h$, denote $\omega(\mathcal T)\subset\mathcal T_h$ the set of all tetrahedra having at least one vertex in $\mathcal T$.

\begin{assumption}\label{Ass1}
For any $T\in \mathcal T_h^\Gamma$ we assume that the  {sets} $W^{ {\pm}}(T)=\mathcal {T}_{h,i}^\pm\cap\omega(\omega(T))$  {are} not empty.
\end{assumption}

We note that this assumption can be weaken by allowing in $W^{\pm}(T)$ neighbors of $T$ of degree $L$, with some finite and mesh independent $L\ge 2$.

Given $T\in\mathcal T_h^\Gamma$, we associate arbitrary but fixed $K_T^{ {\pm}}\in W^{ {\pm}}(T)$, which can be reached from $T$ by crossing faces in  {$\mathcal E_h^{\Gamma,\pm}$}. More precisely, there exists simplexes $T=K_1^{ {\pm}},K_2^{ {\pm}},\ldots,K_M^{ {\pm}}=K_T^{ {\pm}}$ with $K_j^{ {\pm}}\in\mathcal T_h^\Gamma$ for $j<M$. The number $M$ is uniformly bounded and only depends on the shape regularity of the mesh. Moreover, note that by \eqref{eq:minang} there exists a constant $c$ only depending on the shape regularity constant $\kappa$ such that
\[
\frac{h_T}{c}\le h_{K_T^\pm}\le ch_T.
\]
For $T\in\mathcal  {T}_{h,i}^\pm$, we set $K_T^{ {\pm}}=T$.

\begin{assumption}\label{Ass2}
Let $F\in\mathcal E_h^{\Gamma, {\pm}}$, with $F=\partial T_1\cap \partial T_2$. Assume $K_{T_2}^{\pm}$ can be reached from $K_{T_1}^{\pm}$ by crossing a finite, independent of $h$, number of faces of tetrahedra from $\mathcal  {T}_{h,i}^\pm$.
\end{assumption}

The following assumption is also a type of inf-sup condition but restricted to  {interior elements in two phases, i.e. those lying inside $\Omega_{h,i}^\pm$}.
\medskip

\begin{assumption}\label{Ass3} There exists a constant $\beta >0$ such that
\[
\beta |q^\pm|_{H_{h,i}^{1,\pm}}  \le \sup_{\vv v \in  {\vv V_h^{\rm bulk}} \cap \vv H_0^1(\Omega_{h, i}^\pm)} \frac{(\div \, \vv v, q^\pm)_{\Omega_{h,i}^\pm}}{\|\vv v\|_{1, \Omega_{h, i}^\pm}} \quad \text{ for all } q \in  {M_h^{\rm bulk}},
\]
where
\begin{equation*}
|q^\pm|_{H_{h,i}^{1,\pm}}^2= \sum_{T \in \mathcal{T}_{h, i}^\pm} h_T^2 \|\nabla q^\pm\|_{L^2(T)}^2+ \sum_{e \in \mathcal{E}_{h, i}^\pm} h_e \| \jump{q^\pm}\|_{e}^2.
\end{equation*}
\end{assumption}

Examples of pair of spaces $ {\vv V_h^{\rm bulk}}$--$ {M_h^{\rm bulk}}$ that satisfy the Assumption~\ref{Ass3}  can be found in \cite[Section 6]{guzman-olshanskii};  {they include $P_{k+1}-P_k$, $k\ge 1$, and   $P_{k+d}-P_{k}^{\rm disc}$ for $k\ge0$, $\Omega\subset\mathbb{R}^d$, $d=2,3$, and several other elements.} In particular, if a pair $ {\vv V_h^{\rm bulk}}$--$ {M_h^{\rm bulk}}$ satisfies these assumptions, we have the following result.
\begin{theorem}\label{thm:infsupgo}
Suppose Assumptions \ref{Ass1}--\ref{Ass3} hold. There exists a constant $\theta>0$, independent of $h$ and $q$, and a constant $h_0>0$ such that for all $q\in L_0^2(\Omega_{h,i}^\pm)\cap M_h^{\rm bulk}$ and $h\le h_0$, we have
 \[
 \theta\|q\|_{\Omega_{h,i}^\pm}\le \sup_{\vv v\in \vv V_{h}^\Gamma \cap \vv H_0^1(\Omega_{h,i}^\pm)}\frac{(\div\,\vv v,q)_{\Omega_{h,i}^\pm}}{\|\vv v\|_{1,\Omega_{h,i}^\pm}}.
 \]
\end{theorem}
\begin{proof}
	See \cite[Theorem 1]{guzman-olshanskii}.
\end{proof}

\smallskip

We will also need trace and inverse inequalities which can be found, for example, in \cite{guzman-olshanskii}.
\begin{lemma}\label{traceinverse}
Let $T \in \mathcal{T}_h$, then it holds
 \begin{alignat}{2}
 \|v\|_{L^2(\partial T)} & \le C (h_T^{-1/2}\|v\|_{T}+h_T^{1/2} \|\nabla v\|_{T}^2), \quad && \text{ for all } v \in H^1(T),  \label{trace1} \\
\|v\|_{L^2(T \cap \Gamma)} & \le C (h_T^{-1/2}\|v\|_{T}+h_T^{1/2} \|\nabla v\|_{T}^2), \quad && \text{ for all } v \in H^1(T), \label{trace2} \\
 \|v\|_{L^2(\partial T)} & \le C h_T^{-1/2}\|v\|_{T} \quad && \text{ for all } v \in P_k(T),  \label{inverse1} \\
 \|v\|_{L^2(T \cap \Gamma)} & \le C h_T^{-1/2}\|v\|_{T} \quad && \text{ for all } v \in P_k(T). \label{inverse2}
 \end{alignat}
with a constant $C$ independent of $T$ and on how $\Gamma$ intersects $T$.
\end{lemma}

\subsection{Discrete Extension Operator}
\label{sec:discextop}
As in the continuous setting, a chief tool will be an extension operator. In this section we provide a discrete analogue of the extension operator in \eqref{eq:extprop}.  Widlund \cite{widlund1987extension} provided a discrete extension operator when the mesh fits the interface. For non-fitted meshes, as is the case here, a discrete extension operator can be found in \cite{bgss-2017} for piecewise-linear finite element functions and smooth interface.

 In Lemma~\ref{LemmaExt} below, we prove the result for unfitted meshes   and finite elements of \emph{arbitrary degree}
that admit the existence of a  \emph{local nodal basis}. Let us make this assumption precise.
The velocity bulk space is the space of vector functions, i.e. $\mathbf{V}_h^{\rm bulk}=\otimes_{j=1}^d V_{j,h}$. For each component space $V_{j,h}$  we assume that (i) there is a set of points (nodes) $\mathcal{N}(\mathcal{T}_h)=\{ \vv y_1, \ldots, \vv y_\ell\}$ such that $v\in V_{j,h}$ is uniquely determined by $v(\vv y_i)$ for $1 \le i \le \ell$; (ii) for each $T$  there exists a local subset $\mathcal{N}(T)=\{ \vv y \in \mathcal{N}(\mathcal{T}_h)\,:\,\vv y \in \overline{T} \}$ such that $v|_T$ is uniquely  determined by the values there; and (iii) if $\Phi: T\to \widehat{T}$ is the affine mapping to the reference simplex, then  $\Phi(\mathcal{N}(T))$ is independent of $T\in\mathcal{T}_h$.  The existence of the {local nodal basis} is, of course, standard if $\mathbf{V}_h^{\rm bulk}=\vv W_h^k$ for some integer $k\ge 1$. Moreover, we assume only  \emph{Lipschitz regularity} of the interface.

\begin{lemma}[Finite element extension]\label{LemmaExt} Assume $\Gamma$ is \emph{Lipschitz},  the meshes $\{\mathcal{T}_h\}$ are shape regular and satisfy Assumptions \ref{Ass1},  and assume $\mathbf{V}_h^{\rm bulk}$ has a {local nodal basis}. There exists an  extension operator $E_h:\vv V_h^+\to \vv V_h^{\rm bulk}$ with the following properties:
\begin{itemize}
	\item[a)] $E_h\vv v_h=\vv v_h$  on $\Omega_h^+$,
	\item[b)] There exists $C>0$, independent of $h$  and position of $\Gamma$ against the underlying mesh, such that for all $\vv v_h \in \vv V_h^+$,
	\[
	\|E_h\vv v_h\|_{1,\Omega}\le C\|\vv v_h\|_{1,\Omega_h^+}.
	\]
\end{itemize}
\end{lemma}
\begin{proof}
  The proof  is given in the Appendix.
\end{proof}

%%%%%%%%%%%%%%%%%%
\subsection{The discrete variational formulation}
In this section, we define the discrete counterpart of \eqref{eq:mixedform}. The jumps over the interface are enforced weakly, and a term is added to enforce the symmetry of the bilinear form $a$. A discrete variational analogue of \eqref{eq:mixedform} is given by the problem of finding $(\vv u_h,p_h)\in \vv V_h\times M_h$ such that
\begin{equation}\label{eq:mixeddiscform}\begin{array}{cccccc}a_h(\vv u_h,\vv v_h) & + & b_h(\vv v_h,p_h) & = & \vv F_h(\vv v_h) & \forall\,\vv v_h\in \vv V_h,\\   b_h(\vv u_h,q_h) & - & J_h(p_h,q_h)  & = & G_h(q_h) & \forall\,q_h\in M_h, \end{array}
\end{equation}
 {with bilinear forms given below}.
 {For all} $\vv u_h=(\vv u_h^-,\vv u_h^+),$ $\vv v_h=(\vv v_h^-,\vv v_h^+)\in \vv V_h$ and $p_h=(p_h^-,p_h^+)$, $q_h=(q_h^-,q_h^+)\in M_h$,  {we define}
\begin{alignat*}{1}
a_h(\vv u_h,\vv v_h)&:=\big(\nu^-D(\vv u_h^-), D(\vv v_h^-)\big)_{\Omega^-}+\big(\nu^+D(\vv u_h^+),D(\vv v_h^+)\big)_{\Omega^+}+\sum_{T\in\mathcal T_h^\Gamma}\frac{\gamma}{h_T}\nu^-\big(\jump{\vv u_h}, \jump{\vv v_h} \big)_{T_\Gamma}\\
&\qquad -(\nu^-D(\vv u_h^-)\vv n, \jump{\vv v_h})_{\Gamma}-(\nu^-D(\vv v_h^-)\vv n, \jump{\vv u_h})_{\Gamma}+\vv J_h(\vv u_h,\vv v_h),
\end{alignat*}
 {
with
\[
\vv J_h(\vv u_h,\vv v_h)=\vv J_h^{-}(\vv u_h,\vv v_h)+ \vv J_h^{+}(\vv u_h,\vv v_h),\quad
\vv J_h^{\pm}(\vv u_h,\vv v_h)=
\sum_{\ell=1}^{s}|e|^{2\ell -1} \sum_{e\in\mathcal E_h^{\Gamma,\pm}} \gamma^\pm_{\vv u}\nu^\pm(\jump{\partial_n^\ell \vv u_h^\pm}, \jump{\partial_n^\ell \vv v_h^\pm})_e,
\]
}
\[
b_h(\vv v_h,q_h):=-\big(q_h^-,\div\,\vv v_h^-)_{\Omega^-}- \big( q_h^+,\div\,\vv v_h^+\big)_{\Omega^+}+\big( q_h^-,\jump{\vv v_h}\cdot \vv n)_{\Gamma},
\]
and
\[
 J_h(p_h,q_h)=J_h^-( {p_h,q_h})+J_h^+( {p_h,q_h}),\quad
J_h^\pm( {p_h,q_h}):=\frac{ {\gamma_p^\pm}}{\nu^\pm}\sum_{e\in\mathcal E_h^{\Gamma,\pm}}\sum_{\ell = 0}^{k_p}|e|^{2\ell+1}(\jump{\partial_n^\ell p_h^\pm},\jump{\partial_n^\ell q_h^\pm})_e.
\]
Stabilization parameters $\gamma^\pm_p$, $\gamma^\pm_{\vv u}$ and $\gamma$ are all assumed to be independent of $\nu$, $h$ and position of $\Gamma$ against the underlying mesh. Parameter $\gamma$  needs to be large enough to provide the bilinear form $a(\cdot,\cdot)$ with coercivity.  For the purpose of analysis, we set $\gamma^\pm_p=\gamma^\pm_{\vv u}=1$. In practice, these parameters can be tuned for better numerical performance (see section~\ref{sec5} for numerical examples) and the analysis below remains valid if all $\gamma^\pm_p$ and $\gamma^\pm_{\vv u}$ are $O(1)$ parameters.

The right-hand side $\vv F_h\in \vv V_h^*$ will be defined later on, and we assume $G_h\in M_h^*$ is given by
\[
G_h(q_h):=G_h^-(q_h)+G_h^+(q_h),\quad\text{such that}\quad G_h^\pm(q_h)=G_h^\pm(q_h^\pm).
\]
It is straightforward to check that the norm of linear bounded functions $G_h^\pm$ can be expressed in terms of $M_h^\pm$ spaces. More precisely, it holds
\[
\|G_h^\pm\|_{-1,h}=\sup_{q_h^\pm\in M_h^\pm}|G_h^\pm(q_h^\pm)|/\|q_h^\pm\|_{\pm,h}.
\]
Also, we assume $G_h$ satisfies $G_h(1)=0$.  {In particular, this implies that if the second equation of \eqref{eq:mixeddiscform} is satisfied by $q_h\in M_h$, then it is also satisfied by $q_h+c$ for any constant function $c$.}

\subsection{Well-posedness of the discrete scheme}
 We are now interested in a finite element counterpart of the a priori estimate \eqref{eq:contaim}. More precisely, for the solution $(\vv u_h,p_h)$ of  \eqref{eq:mixeddiscform}, we shall prove the following stability result:
\begin{equation}
\label{eq:discaim}\|\nu D(\vv u_h)\|_{\Omega}+\|p_h\|_{\Omega}\le C\big(\|\vv F_h\|_{-1,h}+\nu^-\|G_h^-\|_{-1,h}+\nu^+\|G_h^+\|_{-1,h}\big)
\end{equation}
with  a constant $C>0$ independent of $\nu^{\pm}$, $h$ and the position of $\Gamma$ in the bulk mesh. The proof will largely follow the main steps made in section~\ref{sec2}. We first prove specific estimates for discrete pressure in $\Omega$ and subdomains.  The continuity of the bilinear forms $a_h$ and $b_h$, and the coercivity of $a_h$ in $\vv V_h$ will help us with energy estimates, which due to $\nu^-\le \nu^+$ yield the desired control of  $\nu D(\vv u_h)$, pressure and stabilization terms in $\Omega^-$.
To extend these estimates to $\Omega^+$, a crucial result about extension of finite element functions is stated (with its proof moved to the Appendix section). The result is then applied to gain control of finite element viscous stresses and pressure in  $\Omega^+$. Define the  {natural energy norm for} $a_h$ by
\begin{equation}\label{eq:defnormVh}
\|\vv v_h\|_{\vv V_h}^2:=\|\nu^{\frac12} \,D(\vv v_h)\|^2_{\Omega}+\sum_{T\in\mathcal T_h^\Gamma}\frac{\|\sqrt{\nu^-}\,\jump{\vv v_h}\|_{T_\Gamma}^2}{h_T}+ {\vv J_h(\vv v_h,\vv v_h)},\quad  \text{for all}~\vv v_h=(\vv v_h^-,\vv v_h^+)\in \vv V_h.
\end{equation}
We will need the following technical lemma which is essentially found in \cite[Lemma 5.1]{larson-stabfict}.  {The main difference is the use of Korn's inequality and projections onto the rigid body motions instead of the Poincar\'e inequality and projection onto the constants.}
\begin{lemma}
\label{lem:overlap}
There exists $C>0$, independent of $h$ and $\nu^\pm$, such that for every $q_h\in M_h$, it holds

\begin{equation}\label{eq:overlapph}
\|q_h^\pm\|_{\Omega_h^\pm}^2\le C\bigg(\|q_h^\pm\|_{ {\Omega^\pm_{h,i}}}^2+
 {\nu^-J_h^-(q_h^-,q_h^-)+\nu^-J_h^+(q_h^+,q_h^+)\bigg),}
\end{equation}
and for all $\vv v_h\in \vv V_h$,
\begin{equation}\label{eq:overlapuh}
\|D(\vv v_h^\pm)\|_{\Omega_h^\pm}^2\le C\bigg(\|D(\vv v_h^\pm)\|_{ {\Omega^\pm_{h,i}}}^2+ {\frac1{\nu^-}\vv J_h^-(\vv v_h^-,\vv v_h^-)+ \frac1{\nu^+}\vv J_h^+(\vv v_h^+,\vv v_h^+)}\bigg)\le C\|\vv v_h\|_{1,h}^2.
\end{equation}
\end{lemma}

The following result discusses the continuity and coercivity of the bilinear form $a_h$. We omit the proof since they are by now standard and can easily be proved using Lemma \ref{traceinverse}. See, for example, similar results in \cite{guzman-olshanskii}.
\begin{lemma}
Let $\vv v_h,\,\vv w_h \in \vv V_h$ and $q_h\in M_h$. Then, there exists $C>0$ and $h_0$, independent of $h$ and $\nu^\pm$, $\vv v_h$ and $\vv w_h$, such that for all $h\le h_0$ it holds
\[
|a_h(\vv v_h,\vv w_h)|\le C\|\vv v_h\|_{\vv V_h}\|\vv w_h\|_{\vv V_h},
\]
\begin{equation}\label{boundah}
|a_h(\vv v_h, \vv w_h)| \le C \| \vv v_h\|_{1, h, \nu}  \| \vv w_h\|_{1, h}.
\end{equation}
Additionally, for all $\vv v \in  [\vv H^{s+1}(\Omega_h^-) \times \vv H^{s+1}(\Omega_h^+)]\times \vv V_h$ we have
\begin{equation}\label{boundahsmooth}
|a_h(\vv v, \vv w_h)| \le C \| \vv v\|_{1, h, \nu,\star}  \| \vv w_h\|_{1, h}.
\end{equation}
Finally, there exists $\alpha>0$, independent of $h$ and $\nu^\pm$, such that
\[
\alpha\|\vv v_h\|_{\vv V_h}^2\le a_h(\vv v_h,\vv v_h)\quad\forall\,\vv v_h\in \vv V_h.
\]
\end{lemma}

The following result is the discrete analogue of Lemma \ref{lemma:ppm}.

\begin{lemma}\label{Lem688}
	Let $\vv u_h\in\vv V_h$ and $p_h\in M_h$ solve \eqref{eq:mixeddiscform}. Then, there exists $C,h_0>0$, depending only on $\Omega$ and $\Gamma$, such that for $h\le h_0$ it holds
	\begin{equation}\label{eq:phavgpm}
	\|p_h^\pm-\mathrm{avg}_{\Omega^\pm_{h,i}}(p_h^\pm)\|_{\Omega^\pm}\le C\big(\|\nu^\pm D(\vv u_h^\pm)\|_{\Omega^\pm}+(\nu^\pm J_h(p_h^\pm,p_h^\pm))^{1/2}+\|\vv F_h\|_{-1,h}\big),\end{equation}
	and
	\begin{equation}\label{eq:phinfsup}
	\|p_h\|_{\Omega}\le C\big(\|\vv u_h\|_{1, h, \nu}+ (\nu^- J_h^-(p_h^-,p_h^-))^{1/2}+(\nu^+ J_h^+(p_h^+,p_h^+))^{1/2}+\|\vv F_h\|_{-1,h}\big).
	\end{equation}
\end{lemma}

\begin{proof}
Let $q_h^\pm:=p_h^\pm-\mathrm{avg}_{\Omega_{h,i}^\pm}(p_h^\pm)$. Noting that $J_h^\pm(p_h^\pm,p_h^\pm)=J_h(q_h^\pm,q_h^\pm)$ and employing  {\eqref{eq:overlapph} from} Lemma~\ref{lem:overlap}, we get
\begin{equation}\label{eq:qhavg}C\|q_h^\pm\|_{\Omega^\pm}^2\le \|q_h\|_{\Omega_{h,i}^\pm}^2+\nu^\pm J_h^\pm(p_h^\pm,p_h^\pm).\end{equation}
On the other hand, since $q_h^\pm\in L_0^2(\Omega_{h,i}^\pm)$,  Theorem~\ref{thm:infsupgo} implies that there exist $\theta,h_0>0$, depending only on $\Omega$, such that for all $h\le h_0$, it holds
\begin{equation}\label{eq:infsup}
\theta\|q_h^\pm\|_{\Omega^\pm_{h,i}}\le \sup_{\vv w_h^\pm \in  \vv V_h^{\rm bulk} \cap \vv H_0^1(\Omega_{h,i}^\pm)}\frac{(\div \, \vv w_h^{ {\pm}}, q_h^\pm)_{\Omega_{h,i}^\pm}}{\|\vv w_h^\pm\|_{1,\Omega_{h,i}^\pm}}
=\sup_{\vv w_h^\pm \in \vv V_h^{\rm bulk} \cap \vv H_0^1(\Omega_{h,i}^\pm)}\frac{(\div \, \vv w_h^\pm,p_h^\pm)_{\Omega_{h,i}^\pm}}{\|\vv w_h^\pm \|_{1,\Omega_{h,i}^\pm}}.
\end{equation}
However, for $\vv w_h^\pm \in  {\vv V_h^{\rm bulk}} \cap \vv H_0^1(\Omega_{h,i}^\pm)$ we notice equalities
\begin{equation}\label{aux709}
\frac{(\div \, \vv w_h^\pm,p_h^\pm)_{\Omega_{h,i}^\pm}}{\|\vv w_h\|_{1,\Omega_{h,i}^\pm}}= \frac{b_h( \vv v_h,p_h)}{\|\vv v_h\|_{1,h}}=\frac{\vv F_h(\vv v_h)-a_h(\vv v_h, \vv v_h)}{\|\vv v_h\|_{1,h}},
\end{equation}
where $\vv v_h=(\vv w_h^-, 0)$ or $\vv v_h=(0, \vv w_h^+)$. Because  $\vv v_h$ is supported on $\Omega_{h,i}^\pm$ we have
\begin{equation}\label{aux713}
a_h(\vv u_h,\vv v_h):=\big(\nu^-D(\vv u_h^-), D(\vv v_h^-)\big)_{\Omega^-}+\big(\nu^+D(\vv u_h^+),D(\vv v_h^+)\big)_{\Omega^+}+ \vv J_h(\vv u_h,\vv v_h),
\end{equation}
With the help of the Cauchy-Schwarz inequality, inverse estimates and Korn's inequality we obtain
\begin{equation}\label{aux717}
|a_h(\vv u_h,\vv v_h)| \le C \|\nu^\pm D(\vv u_h^\pm)\|_{\Omega^\pm} \|\vv v_h\|_{1,h}.
\end{equation}
Using \eqref{aux709}--\eqref{aux713} in \eqref{eq:infsup},  we arrive at
\begin{equation}\label{aux721}
\|q_h^\pm\|_{\Omega^\pm}\le C\big(\|\nu^\pm D(\vv u_h^\pm)\|_{\Omega^\pm}+(\nu^\pm J_h^\pm(p_h^\pm,p_h^\pm) )^{1/2}+\|\vv F_h\|_{-1,h}\big),
\end{equation}
which leads to \eqref{eq:phavgpm}.

 In order to prove  {\eqref{eq:phinfsup}}, we  {consider $\alpha_h=p_h-q_h$}
 and  observe that $\|\alpha_h\|_{L^\infty(\Omega)} \le C \|p_h\|_{\Omega}$. Moreover,  a simple calculation shows that
 \begin{equation*}
 (\alpha_h, 1)_{\Omega}=\frac{|\Omega^+|}{|\Omega_{h,i}^+|} (p_h^+, 1)_{\Omega^+}+\frac{|\Omega^+|}{|\Omega_{h,i}^+|}(p_h^-, 1)_{\Omega^-}
 {-}\frac{|\Omega^+|}{|\Omega_{h,i}^+|} (p_h^+, 1)_{\Omega^+\setminus \Omega_{h,i}^+} {-}\frac{|\Omega^-|}{|\Omega_{h,i}^-|} (p_h^-, 1)_{\Omega^-\setminus \Omega_{h,i}^-}.
 \end{equation*}
 Hence, using that $(p_h^-, 1)_{\Omega^-}+(p_h^+, 1)_{\Omega^+}=0$ and that $ {|\Omega_h^\pm\setminus\Omega_{h,i}^\pm|} =|\Omega_h^\pm|-|\Omega_{h,i}^\pm| \le C \, h$ we have
\begin{equation}\label{aux222}
\begin{split}
| (\alpha_h, 1)_{\Omega}|& {\le C \left(|(p_h^+, 1)_{\Omega^+\setminus \Omega_{h,i}^+}|+|(p_h^-, 1)_{\Omega^-\setminus \Omega_{h,i}^-}|\right)+ \left|\frac{|\Omega^+|}{|\Omega_{h,i}^+|}-\frac{|\Omega^-|}{|\Omega_{h,i}^-|}\right|  |(p_h^+,1)_{\Omega^+}|}\\
&\le  {C (\|p_h^+\|_{\Omega^+} \|1\|_{\Omega^+\setminus \Omega_{h,i}^+}+\|p_h^-\|_{\Omega^-} \|1\|_{\Omega^-\setminus \Omega_{h,i}^-}+ h\|p_h^+\|_{L^1(\Omega^+)}) }
\le C \,  {h^{\frac12}} \|p_h\|_{\Omega}.
\end{split}
\end{equation}

In the case \ref{when} $ {M_h^{\rm bulk}}= M_{h, \rm disc}$ we let $r_h \in  {M_h^{\rm bulk}}$ to be the $L^2$ projection of $\alpha_h$ onto piecewise constants with respect to the mesh $\mathcal{T}_h$. In the case, $ {M_h^{\rm bulk}}= M_{h, \rm cont}$ we let $r_h \in {M_h^{\rm bulk}}$ be the continuous piecewise linear function such that $r_h(x) =\alpha_h^+(x)$ if $x$ is a vertex and $x \in  \Omega^+$, and $r_h(x) =\alpha_h^-(x)$ if $x$ is a vertex and $x \in  \overline{\Omega^-}$. In either case, $r_h|_{\Omega_{h,i}^\pm}= \alpha_h^\pm$.

 Recalling the notation $\omega_h=\Omega\setminus(\Omega_{h,i}^-\cup\Omega_{h,i}^+)$, we then note that
\begin{equation}\label{aux333}
\|r_h -\alpha_h\|_{\Omega}= \|r_h-\alpha_h\|_{\omega_h} \le C h^{ {\frac12}}  \|r_h-\alpha_h\|_{L^\infty(\omega_h)} \le   C h^{ {\frac12}}  \|\alpha_h\|_{L^\infty(\omega_h)} \le  C h^{ {\frac12}}  \|p_h\|_{\Omega}.
\end{equation}
We let $\tilde{r}_h=r_h-\text{avg}_{\Omega}(r_h)$. From \eqref{aux222} and \eqref{aux333} we have that
\begin{equation}\label{aux444}
\|\tilde{r}_h-r_h\|_{\Omega} \le C h^{ {\frac12}} \|p_h\|_{\Omega} \quad \text{ which implies } \quad \|\tilde{r}_h-\alpha_h\|_{\Omega} \le C h^{ {\frac12}} \|p_h\|_{\Omega}.
\end{equation}
Assumption~\ref{Ass0} provides us with $\vv v_h  \in  {\vv V_h^{\rm bulk}}$ such that
\begin{equation}\label{aux555}
\xi \|\tilde{r}_h \|_{\Omega} \le (\div \, \vv v_h, \tilde{r}_h)_{\Omega}\quad\text{and}\quad{\|\vv v_h\|_{1, \Omega}=1}.
\end{equation}

Let $\vv w_h \in \vv V_h$ be given by $\vv w_h=(\vv v_h|_{\Omega_h^-}, \vv v_h|_{\Omega_h^+})$. It holds $\|\vv w_h\|_{1,h} \le C \|\vv v_h\|_{1, \Omega}$. To verify the last inequality, we note that the first term in the definition of $\|\vv w_h\|_{1,h}$  vanish, while the second jump term can be estimated with the help of the finite element trace and inverse inequalities.
Hence, we get
\begin{alignat*}{1}
 (\div \, \vv v_h, \tilde{r}_h)_{\Omega}&=(\div \, \vv v_h, p_h^-)_{\Omega^-} + (\div \, \vv v_h, p_h^+)_{\Omega^+}+(\div \, \vv v_h, \tilde{r}_h^-  -p_h^-)_{\Omega^-} + (\div \, \vv v_h, \tilde{r}_h^+-p_h^+)_{\Omega^+} \\
 &= a_h(\vv u_h,\vv w_h)  -  \vv F_h(\vv w_h) +(\div \, \vv v_h, \tilde{r}_h^- - p_h^-)_{\Omega^-} + (\div \, \vv v_h, \tilde{r}_h^+-p_h^+)_{\Omega^+},
 \end{alignat*}
 where we used the first equation of \eqref{eq:mixeddiscform}.
Thanks to \eqref{boundah} and \eqref{aux555} we have
 \begin{alignat}{1}
\xi \|\tilde{r}_h \|_{\Omega} &\le   C (\| \vv u_h\|_{1,h, \nu}+ \| \vv F_h\|_{-1,h}+\|\tilde{r}_h^--p_h^-\|_{\Omega^-}+ \|\tilde{r}_h^+-p_h^+\|_{\Omega^+})  \nonumber \\
 &\le  C (\| \vv u_h\|_{1,h, \nu}+ \| \vv F_h\|_{-1,h}+\|\tilde{r}_h-\alpha_h\|_{\Omega}+ \|q_h^-\|_{\Omega^-}+ \|q_h^+\|_{\Omega^+}) . \label{aux11101}
\end{alignat}
We then use the triangle inequality, \eqref{aux721}, \eqref{aux444}, and  \eqref{aux11101} to get
\begin{alignat*}{1}
\|p_h\|_{\Omega} &\le  \|q_h^-\|_{\Omega^-}+\|q_h^+\|_{\Omega^+}+ \|\alpha_h-\tilde{r}_h\|_{\Omega}+ \|\tilde{r}_h\|_{\Omega} \\
&\le  C (\| \vv u_h\|_{1,h, \nu}+ \| \vv F_h\|_{-1,h} + (\nu^- J_h^-(p_h^-,p_h^-))^{1/2}+(\nu^+ J_h^+(p_h^+,p_h^+))^{1/2}+ h^{ {\frac12}} \|p_h\|_{\Omega}) .
\end{alignat*}
The result now follows after taking  $h_0$ small enough.

\end{proof}
In order to prove \eqref{eq:discaim}, we start by obtaining an estimate for $\|\nu^-D(\vv u_h^-)\|_{\Omega^-}$ using the energy norm \eqref{eq:defnormVh}.

%%%%%%%%%%%%%%%%%%%%

\begin{lemma}
	\label{lem:energyh}
	Let $ (\vv u_h,p_h)\in V_h\times M_h$ be a solution of \eqref{eq:mixeddiscform}. Then, there exists $C>0$, independent of $h$ and $\nu^\pm$, such that
\begin{equation}\label{eq:rhomuh}
		\begin{split}
		&\|\nu^-D(\vv u_h^-)\|_{\Omega^-}^2+ {\nu^-\vv J_h^-(\vv u_h^-,\vv u_h^-)}
		+\nu^-J_h^-(p_h^-,p_h^-) \\
		&\qquad\qquad\qquad\le C\big(\|\vv F_h\|_{-1,h}+\nu^-\|G_h^-\|_{-1,h}+\nu^+\|G_h^+\|_{-1,h}\big)^2.
		\end{split}
	\end{equation}
\end{lemma}
\begin{proof}
	We use the first and second equation of \eqref{eq:mixeddiscform} with $\vv v_h=\vv u_h$ and $q_h=p_h$, respectively, and the coercivity of $a_h$  {for $\gamma$ large enough} to get
	\begin{equation}\label{eq:nrgnum}
	\alpha\|\vv u_h\|_{\vv V_h}^2+J_h(p_h,p_h)\le a_h(\vv u_h,\vv u_h)+J_h(p_h,p_h)\le|\vv F_h(\vv u_h)|+|G_h(p_h)|,
	\end{equation}
 {with some $\alpha>0$ independent of $\nu$ and $h$.}
	By definition of the norms $\|\cdot\|_{ {1,h}}$ and $\|\cdot\|_{\vv V_h}$, and since $\nu^-\le\nu^+$, we get
\begin{equation}\label{aux808}
	|\vv F_h(\vv u_h)|\le\|\vv F_h\|_{-1,h}\|\vv u_h\|_{1,h}\le \frac{\|\vv F_h\|_{-1,h}\|\vv u_h\|_{\vv V_h}}{\sqrt{\nu^-}},
\end{equation}
so that it only remains to estimate $|G_h(p_h)|$. In order to do this, we will decompose the expression $G_h(p_h)=G_h^-(p_h^-)+G_h^+(p_h^+)$ into three terms:
	%\begin{equation}\label{eq:decompgh}
	\[
	G_h(p_h) = G_h^-(p_h^--\alpha_h^-)+G_h^+(p_h^+-\alpha_h^+)+G_h(\alpha_h),
	\]
	%\end{equation}
	where $\alpha_h^\pm:=\mathrm{avg}_{\Omega_{h,i}^\pm}(p_h^\pm)$.
	Then, using \eqref{eq:phavgpm} and $\nu^- \le \nu^+$ we get
	\[
	\big|G_h^\pm(p_h^\pm-\alpha_h^\pm)\big|\le C\sqrt{\nu^\pm}\|G_h^\pm\|_{\pm,h}\big(\|\sqrt{\nu^\pm} D(\vv u_h^\pm)\|_{\Omega^\pm}+J_h^\pm(p_h^\pm,p_h^\pm)^{1/2}\big)+C\|G_h^\pm\|_{\pm,h}\|\vv F_h\|_{-1,h}.
	\]
	On the other hand,
	\begin{equation*}
	|G_h(\alpha_h)|=|G_h^+(\alpha_h^+)+G_h^-(\alpha_h^-)|=|G_h^+(\alpha_h^+-\alpha_h^-)| \le \|G_h^+\|_{+, h} \|\alpha_h^+-\alpha_h^-\|_{\Omega^+} \le C  \|G_h^+\|_{+, h} \|p_h\|_{\Omega},
	\end{equation*}
	where we used our assumption that $G_h^+(1)=-G_h^-(1)$ .
	Therefore, we have
	\begin{equation}\label{aux828}
	|G_h(p_h)|  \le C ( \sqrt{\nu^-}\|G_h^-\|_{-1,h}+ \sqrt{\nu^+}\|G_h^+\|_{-1,h}) (\|\vv u_h\|_{\vv V_h}+ J_h(p_h, p_h)^{1/2}+ \frac{1}{\sqrt{\nu^-}}{\|\vv F_h\|_{-1,h}}+ \frac{1}{\sqrt{\nu^+}}{\| p_h\|_{\Omega}}),
	\end{equation}
	where we used that $\nu^- \le \nu^+$. After re-scaling by $\frac{1}{\sqrt{\nu^+}}$ the estimate \eqref{eq:phinfsup} yields for the last term on the right-hand side of \eqref{aux828} the bound,
	\[
	\frac{1}{\sqrt{\nu^+}}{\| p_h\|_{\Omega}} \le C (\|\vv u_h\|_{\vv V_h}+ J_h(p_h, p_h)^{1/2}+ \frac{1}{\sqrt{\nu^-}}{\|\vv F_h\|_{-1,h}}),
	\]
	again due to $\nu^- \le \nu^+$.
Using now \eqref{aux808} and \eqref{aux828} to bound from above the right-hand side of \eqref{eq:nrgnum} leads after some calculations to
	\begin{equation*}
	\alpha\|\vv u_h\|_{\vv V_h}^2+J_h(p_h,p_h) \le  C ( \sqrt{\nu^-}\|G_h^-\|_{-1,h}+ \sqrt{\nu^+}\|G_h^+\|_{-1,h}) (\|\vv u_h\|_{\vv V_h}+ J_h(p_h, p_h)^{1/2}+ \frac{1}{\sqrt{\nu^-}}{\|\vv F_h\|_{-1,h}}).
	\end{equation*}
	Therefore, we have
	\begin{equation*}
	\|\vv u_h\|_{\vv V_h}^2+J_h(p_h,p_h) \le  C ( \sqrt{\nu^-}\|G_h^-\|_{-1,h}+ \sqrt{\nu^+}\|G_h^+\|_{-1,h}  {+} \frac{1}{\sqrt{\nu^-}}{\|\vv F_h\|_{-1,h}})^2.
	\end{equation*}
	The result now follows by multiplying both sides by $\nu^-$ and using that $\nu^- \le \nu^+$.
	\end{proof}
As an immediate consequence of Lemma~\ref{lem:energyh} and \eqref{eq:phavgpm}, we have the following result.
\begin{lemma}
	%\label{lem:phmavg}
	Let $\vv u_h\in \vv V_h$ and $p_h\in M_h$ solve \eqref{eq:mixeddiscform}. Then, there exist $C,h_0>0$, depending only on $\Omega$ and $\Gamma$, such that for $h\le h_0$,
	\begin{equation}
	\label{eq:phavgm}
	\|p_h^--\mathrm{avg}_{\Omega_{h,i}^-}(p_h^-)\|_{\Omega^-}\le C\big(\nu^-\|G_h^-\|_{-1,h}+\nu^+\|G_h^+\|_{-1,h}+\|\vv F_h\|_{-1,h} \big).
	\end{equation}
\end{lemma}

The next result is the discrete analogue of Lemma \ref{lemma:eq:errorflux}.
\begin{lemma}
	\label{lem:nupDuhp}
	Let $\vv u_h\in \vv V_h$ and $p_h\in M_h$ solve \eqref{eq:mixeddiscform}. Then, there exist $C,h_0>0$, depending only on $\Omega$ and $\Gamma$, such that for $h\le h_0$ it holds
	\begin{equation}
	\label{eq:nupDuhp}
	\|\nu^+D(\vv u_h^+)\|_{\Omega^+_h}+(\nu^+J_h^+(p_h^+,p_h^+))^{1/2}\le C\big(\nu^-\|G_h^-\|_{-1,h}+\nu^+\|G_h^+\|_{-1,h}+\|\vv F_h\|_{-1,h}\big).
	\end{equation}
\end{lemma}
\begin{proof}
	We  first  {note that the rigid body motions belong to the velocity finite element space} {, and using Lemma~\ref{LemmaExt}, we consider the discrete extension $E_h:\vv V_h^+\to \vv V_h^{\rm bulk}$ and define $\vv w_h\in \vv V_h^{\rm bulk}$ by}
	%\begin{equation}\label{eq:discext}
	\[
	\vv w_h:=\begin{cases}E_h\vv u_h^+&\text{ if }|\partial\Omega\cap\partial\Omega_h^+|>0,\\E_h\big(\vv u_h^+-\mathcal P_{\Omega_h^+}^{RM}(\vv u_h^+)\big) &\text{ if }|\partial\Omega\cap\partial\Omega_h^+|=0.\end{cases}
	\]
	%\end{equation}
Let $h_0>0$ be sufficiently small such that if $\Omega^+$ is the inclusion, then $\partial\Omega\cap\partial\Omega_h^+=\emptyset$.
	Then, if $\Omega^+$ is not the inclusion, we have by the boundedness of the extension $E_h$ and by Korn's inequality that $\|\vv w_h\|_{1,\Omega}\le \|\vv u_h^+\|_{1,\Omega_h^+}\le C\|D(\vv u_h^+)\|_{\Omega_h^+}$.
On the other hand, if $\Omega^+$ is the inclusion, we have by the boundedness of the discrete extension $E_h$, Korn's inequality and since
\begin{equation}\label{aux756}
D(\vv w_h^+)=D(\vv u_h^+)  \quad \text{ in } \Omega_h^+,
\end{equation}
that $\|\vv w_h\|_{1,\Omega}\le C\|\vv u_h^+-\mathcal P_{\Omega_h^+}^{\mathrm{RM}}(\vv u_h^+)\|_{1,\Omega_h^+}\le C\|D(\vv u_h^+)\|_{\Omega_h^+}$. In both cases, it holds
	\begin{equation}\label{eq:boundext}
	\|\vv w_h\|_{1,\Omega}\le C\|D(\vv u_h^+)\|_{\Omega_h^+}.
	\end{equation}
	
	We let $\vv v_h \in \vv V_h$ be given by $\vv v_h=(\vv w_h|_{\Omega_h^-}, \vv w_h|_{\Omega_h^+})$. By the same arguments as in the proof of Lemma~\ref{Lem688}, one shows $\|\vv v_h\|_{1,h} \le C \|\vv w_h\|_{1,\Omega}$.  First equation from \eqref{eq:mixeddiscform} yields
\begin{alignat}{1}
\nu^+\|D(\vv u_h^+)\|_{\Omega^+}^2+ {\vv J_h^+(\vv u_h,\vv u_h)} = a_h(\vv u_h,\vv v_h)+R=& F_h(\vv v_h)-b_h(\vv v_h, p_h) +R, \label{988}
\end{alignat}
where
\begin{alignat*}{1}
R=-(\nu^-D(\vv u_h^-),D(\vv v_h))_{\Omega^-}-\big(\nu^-D(\vv v_h^-)\vv n^-,\jump{\vv u_h}\big)_{\Gamma}- {\vv J_h^-(\vv u_h,\vv v_h)}.
\end{alignat*}
We bound every term on the right hand side of \eqref{988}.  With the help of \eqref{inverse2} and \eqref{eq:overlapuh}, we get for the last term
\begin{equation}\label{987}
R \le C \Xi\, \|\vv  v_h\|_{1,h},
\end{equation}
where
\begin{equation*}
 \Xi:=\|\nu^-D(\vv u_h^-)\|_{\Omega^-}^2+\sum_{T\in\mathcal T_{h}^\Gamma}\frac{\|\nu^-\jump{\vv u_h}\|_{T_\Gamma}^2}{h_T}
+ \nu^-\vv J_h^-(\vv u_h,\vv u_h)
\end{equation*}
consists of terms already estimated in Lemma~\ref{lem:energyh}.
For the first term on the right-hand side of \eqref{988} we have
\begin{equation}\label{986}
\vv F_h(\vv v_h) \le \|\vv F_h\|_{-1,h} \|\vv v_h\|_{1,h}\le \|\vv F_h\|_{-1,h} \|\vv w_h\|_{1,\Omega}.
\end{equation}
It remains to estimate the second term on the right-hand side of \eqref{988}. Noting that \eqref{aux756} implies $\div\, \vv w_h^+=\div\, \vv u_h^+$ and using $\Omega^+\subset\Omega^+_h$, we get
\begin{equation*}
-b_h(\vv v_h, p_h)=(\div\, \vv v_h^-, p_h^-)_{\Omega^-}+(\div\, \vv v_h^+, p_h^+)_{\Omega^+}=(\div\, \vv w_h^-, p_h^-)_{\Omega^-}+(\div\, \vv u_h^+, p_h^+)_{\Omega^+}.
\end{equation*}
Setting $\alpha_h^-=\text{avg}_{\Omega_{h,i}^-}(p_h^-)$ and using that  $(\div\, \vv w_h, 1)_{\Omega}=0$, we have
\begin{alignat*}{1}
(\div\, \vv w_h^-, p_h^-)_{\Omega^-}&= (\div\, \vv w_h^-, p_h^- -\alpha_h^-)_{\Omega^-}+(\div\, \vv w_h^-, \alpha_h^-)_{\Omega^-}\\
&= (\div\, \vv w_h^-, p_h^- -\alpha_h^-)_{\Omega^-}- \alpha_h^- (\div\, \vv w_h^+,   {1} )_{\Omega^+} \\
&=(\div\, \vv w_h^-, p_h^- -\alpha_h^-)_{\Omega^-}-(\div\, \vv u_h^+, \alpha_h^- )_{\Omega^+} \\
&=(\div\, \vv w_h^-, p_h^- -\alpha_h^-)_{\Omega^-}- G_h^+(\alpha_h^-) \\
&\le \|\vv w_h\|_{1,\Omega^-} \| p_h^- -\alpha_h^-\|_{\Omega^-} + \|G_h^+\|_{-1,h} \|\alpha_h^-\|_{\Omega^+}.
\end{alignat*}
We also have from the second equation of \eqref{eq:mixeddiscform} with $q_h=(0, p_h^+)$
\begin{equation*}
(\div\, \vv u_h^+, p_h^+)_{\Omega^+} = -G_h^+(p_h^+)-J_h^+(p_h^+, p_h^+) \le  \|G_h^+\|_{-1,h}   \|p_h^+\|_{+,h} -J_h^+(p_h^+, p_h^+).
\end{equation*}
Now we see that
\begin{equation}\label{985}
-b_h(\vv v_h, p_h) \le   \|\vv w_h\|_{1,\Omega^-} \| p_h^- -\alpha_h^-\|_{\Omega^-} + \|G_h^+\|_{-1,h}  ( \|p_h^+\|_{+,h}+  \|\alpha_h^-\|_{\Omega^+} ) -J_h^+(p_h^+, p_h^+).
\end{equation}

We combine \eqref{985}, \eqref{986}, \eqref{987}, and \eqref{988} to get
\begin{equation}\label{aux853}
L  \le   (C\Xi+ \|\vv F_h\|_{-1,h}+ \| p_h^- -\alpha_h^-\|_{\Omega^-})  \nu^+ \|\vv w_h\|_{1,\Omega}+ \nu^+\|G_h^+\|_{-1,h}  ( \|p_h^+\|_{+,h}+  \|\alpha_h^-\|_{\Omega^+} ),
\end{equation}
where
\begin{equation*}
L=  {\|\nu^+ D(\vv u_h^+)\|_{\Omega^+}^2+\nu^+\vv J_h^+(\vv u_h, \vv u_h)
+ \nu^+J_h^+(p_h^+,p_h^+).
}
\end{equation*}
{It remains to estimate the solution-dependent terms on the right-hand side of \eqref{aux853}. Thanks to  \eqref{eq:phinfsup} we get}
\begin{alignat*}{1}
\|p_h^+\|_{+,h}+  \|\alpha_h^-\|_{\Omega^+}  & \le (\|p_h\|_{\Omega} +(\nu^+J_h^+(p_h^+, p_h^+) )^{1/2} )  \\
&\le  C \big(\|\vv u_h\|_{1, h, \nu}+ (\nu^- J_h^-(p_h^-,p_h^-))^{1/2}+(\nu^+ J_h^+(p_h^+,p_h^+))^{1/2}+\|\vv F_h\|_{-1,h}\big) \\
 &\le  C \big(\Xi+(\nu^- J_h^-(p_h^-,p_h^-))^{1/2}+\|\vv F_h\|_{-1,h}\big)+C \sqrt{L}.
\end{alignat*}
 {From \eqref{eq:boundext} we conclude} that  $\nu^+ \|\vv w_h\|_{1,\Omega}\le C \nu^+\|D(\vv u_h^+)\|_{\Omega^+}  \le  {C}\sqrt{L}$ and hence we have
\begin{equation*}
\begin{aligned}
L &\le  {C}(\Xi+\| p_h^- -\alpha_h^-\|_{\Omega^-}+(\nu^- J_h^-(p_h^-,p_h^-))^{1/2}+\|\vv F_h\|_{-1,h}+  \nu^+\|G_h^+\|_{-1,h}) \sqrt{L}\\
&\quad+   C \nu^+\|G_h^+\|_{-1,h} {(\|\vv F_h\|_{-1,h} +\Xi+(\nu^- J_h^-(p_h^-,p_h^-))^{1/2})}.
\end{aligned}
\end{equation*}
The result now follows after using \eqref{eq:phavgm} and \eqref{eq:rhomuh}.
\end{proof}
As an immediate consequence of Lemma~\ref{lem:nupDuhp} and \eqref{eq:phinfsup}, we have the following result.
\begin{theorem}
 \label{lem:stabh}
 Let $ (\vv u_h,p_h)\in \vv V_h\times M_h$ be a solution of the finite element method \eqref{eq:mixeddiscform}. Then, there exists $C,h_0>0$, independent of $h$, position of $\Gamma$ in the mesh and $\nu^\pm$, such that for $h \le h_0$, it holds
 %\begin{equation}\label{eq:stabh}
 \[
 \|\nu D(\vv u_h)\|_{\Omega}+\|p_h\|_{\Omega}\le C\big(\nu^-\|G_h^-\|_{-1,h}+\nu^+\|G_h^+\|_{-1,h}+\|\vv F_h\|_{-1,h}\big).
 \]
 %\end{equation}
\end{theorem}

\section{Finite element error estimates}
\label{sec4}
We use the stability results from the previous sections to obtain error estimates for both $\|\nu D(\vv u-\vv u_h)\|_{\Omega}$ and $\|p-p_h\|_{\Omega}$.
We assume that the solutions to the two-phase Stokes problem \eqref{eq:contform} sufficiently smooth in each subdomain. In particular, we assume that  $\vv u^\pm \in \vv H^{s+1}(\Omega^\pm)$ and $p^\pm\in \mathrm{H}^{k_p+1}(\Omega^\pm)$. In such a case there exist extensions of $\vv u$ and $p$, from $\Omega^\pm$ to $\Omega^\mp$, denoted by $\vv u_E^\pm$ and $p_E^\pm$ and with the property $\vv u_E^\pm|_{\Omega^\pm}=\vv u^\pm$ and $p_E^\pm|_{\Omega^\pm}=p^\pm$, such that $\vv u_E^\pm\in \vv H^{s+1}(\Omega)$, $p_E^\pm\in \mathrm{H}^{k_p+1}(\Omega)$,
\[
\|\vv u_E^\pm\|_{s+1,\Omega}\le C\|\vv u^\pm\|_{s+1,\Omega^\pm},\quad\text{and}\quad \|p_E^\pm\|_{k_p+1,\Omega}\le C\|p^\pm\|_{k_p+1,\Omega^\pm},
\]
where $C$ depends only on $\Omega$ and $\Gamma$. Further, we will identify $\vv u^\pm$ and $p^\pm$ with there extensions.

For the error analysis in addition to the augmented norm for the velocity \eqref{augmentednorm} we also define
\begin{equation*}
\|q_h\|_{h,\star}^2:=\|q_h\|_{h}^2+ \sum_{T\in\mathcal T_h^\Gamma} h_T (\|q_h^{-}\|_{T_\Gamma}^2+\|q_h^{+}\|_{T_\Gamma}^2).
\end{equation*}
Multiplying the first equation of \eqref{eq:contform} with $\vv v^\pm$, using that $\jump{\sigma(\vv u,p)\vv n}=\bm{\lambda}$, and taking into account the choice of the weights for the average $\{\vv v\}_\Gamma$ in the definition of $a_h$, we define $\vv F_h$ by
\begin{equation}
\label{eq:defFh}
\vv F_h(\vv v_h):=(\vv f^-,\vv v_h^-)_{\Omega^-}+(\vv f^+,\vv v_h^+)_{\Omega^+}+(\bm{\lambda},\vv v_h^+)_\Gamma.
\end{equation}
Then, we have the main result of this section.

\begin{theorem} \label{ThErr}
Let $(\vv u,p)\in \vv V\times M$ be a solution of \eqref{eq:mixedform} with $\vv f \in \vv L^2(\Omega)$, $\bm{\lambda}\in \vv L^2(\Gamma)$ and $g \in L_0^2(\Omega)$. Furthermore, let $(\vv u_h,p_h)$ be the approximation  that solves \eqref{eq:mixeddiscform} with  {$\vv F_h$ given by \eqref{eq:defFh}} and $G_h^\pm(q_h^\pm)=-(g, q_h^\pm)_{\Omega^\pm}$. Assume that $\vv u^{\pm}  \in \vv H^{s+1}(\Omega^\pm)$ and $p^\pm \in H^{k_p+1}(\Omega^\pm) $. Then there exists a constant $C$ independent of $\nu$, $h$, $\vv u$, $p$ such that
\begin{equation}\label{errEst}
\|\nu D(\vv u-\vv u_h)\|_{\Omega}+\|p-p_h\|_{\Omega} \le C \, \inf_{\vv w_h \in \vv V_h, r_h \in M_h} \left(\|\vv u- \vv w_h\|_{1, h, \nu, \star}+\|p-r_h\|_{h, \star}\right).
\end{equation}
\end{theorem}
\begin{proof}
As discussed above we let $\vv u^\pm $ and $p^{\pm}$ be the extensions of the same functions to the entire $\Omega$. Then,  {by using \eqref{eq:defFh}}, we one easily checks the consistency result:
\begin{alignat*}{2}
a_h(\vv u,\vv v_h)+b_h(\vv v_h,p)&=  \vv F_h(\vv  v_h),  \quad && \text{ for all } \vv v_h \in \vv V_h, \\
b_h(\vv u,q_h)-J_h(p,q_h)&=G_h(q_h),  \quad   && \text{ for all } q_h \in M_h,
\end{alignat*}
Hence, we have for an arbitrary $\vv w_h \in \vv V_h$ and $r_h \in Q_h$ the following consistency result:
\begin{alignat}{2} \label{eq:errorform}
a_h(\vv u_h-\vv w_h,\vv v_h) +  b_h(\vv v_h,p_h-r_h) & =  \vv L_h(\vv v_h) \quad &&   \text{ for all } \,\vv v_h\in \vv V_h,\\
b_h(\vv u_h-\vv w_h,q_h) -  J_h(p_h-r_h,q_h)  & =  Q_h(q_h) \quad &&  \text{ for all } \,q_h\in M_h.
\end{alignat}
 where
 \begin{equation*}
 \vv L_h(\vv v_h):=a_h(\vv u-\vv w_h,\vv v_h)+b_h(\vv v_h,p-r_h),\quad\quad Q_h(q_h):=Q_h^-(q_h)+Q_h^+(q_h),
 \end{equation*}
and
\begin{alignat*}{1}
Q_h^-(q_h)&:=-( q_h^-, \div\,(\vv u^--\vv w_h^-))_{\Omega^{-}}-( q_h^-, \jump{\vv u-\vv w_h}\cdot \vv n^-)_{\Gamma} -J_h^-(p^--r_h^-,q_h^-), \\
Q_h^{+}(q_h)&:= -(q_h^+,\div(\vv u^+-\vv w_h^+))_{\Omega^+}-J_h^+(p^+-r_h^+,q_h^+).
\end{alignat*}
We can easily show, using for example \eqref{boundahsmooth}, the following bound
\begin{equation*}
|\vv L_h(\vv v_h)| \le (\|\vv u- \vv w_h\|_{1, h, \nu, \star}+\|p-r_h\|_{h, \star} ) \|\vv v_h\|_{1, h} ,
\end{equation*}
which implies
\begin{equation*}
\|\vv L_h\|_{-1,h} \le  (\|\vv u- \vv w_h\|_{1, h, \nu, \star}+\|p-r_h\|_{h, \star} ) .
\end{equation*}
Similarly,
\begin{alignat*}{1}
    \nu^\pm |Q_h^{\pm}(q_h)| \le   C (\|\vv u- \vv w_h\|_{1, h, \nu, \star}+\|p-r_h\|_{h, \star} ) \|q_h^\pm\|_{\pm,h}.
\end{alignat*}
Hence,
\begin{equation*}
\nu^- \|Q_h^-\|_{-1,h} +\nu^+ \|Q_h^+\|_{-1,h}\,  {\le}\, C (\|\vv u- \vv w_h\|_{1, h, \nu, \star}+\|p-r_h\|_{h, \star} ).
\end{equation*}
The result follows after applying Theorem~\ref{lem:stabh}  {and triangle inequality}.
\end{proof}

Using \eqref{trace1} and \eqref{trace2} with standard  {interpolation properties of finite element functions and the definition of the norms of the right-hand side of \eqref{errEst}} the next results follows from the theorem.
\begin{corollary}	
	Under the same assumptions as in  Theorem~\ref{ThErr}, it holds:
	\begin{multline}
	\|\nu D(\vv u-\vv u_h)\|_{\Omega}+\|p-p_h\|_{\Omega}\\
	\le C h^{\min\{k_u,k_p+1\}}\big(\|\nu \vv u\|_{k_u+1,\Omega}+ \|p\|_{k_p+1,\Omega}\big)
	+Ch^{k_u}\sum_{\ell=2}^{s+1-k_u} h^{\ell-1} \|\nu\vv u\|_{k_u+\ell,\Omega}\,,
\label{eq:errorrates}
	\end{multline}
 {with a constant $C$ independent of $\nu$, $h$ and the position of the interface in the background mesh.}		The solution norms on the right-hand side of \eqref{eq:errorrates} are the norms in the broken Sobolev spaces  ${H}^\ell(\Omega^-)\times{H}^\ell(\Omega^+)$,
 \begin{equation*}
\label{eq:notation2qh} \|q\|_{\ell,\Omega}^2=\|q\|_{\ell,\Omega^-}^2+\|q\|_{\ell,\Omega^+}^2,\quad\text{for}~q\in{H}^\ell(\Omega^-)\times{H}^\ell(\Omega^+),
\end{equation*}
and similar for vector functions from ${H}^\ell(\Omega^-)^d\times{H}^\ell(\Omega^+)^d$.		
\end{corollary}

\section{Numerical experiments}
\label{sec5}
\subsection*{Example 1} Consider the squared domain $\Omega:=(-1,1)\times(-1,1)$ and the embedded interface $\Gamma:x_1^2+x_2^2=R^2$ for $R=1/\sqrt{\pi}$. We define $\Omega^{-}=\{\vv x\in\mathbb{R}^2\,:\,|\vv x|<R\}$, and $\Omega^{+}=\Omega\setminus\overline{\Omega}^{+}$, and choose the data $\vv f\in\mathbf L^2(\Omega)$ so that the exact solution $(\vv u,p)$ is given for all $\vv x=(x_1,x_2)$ by
\begin{equation}
\label{eq:example1}
\vv u(\vv x)=\begin{cases}\dfrac{R^2-|\vv x|^2}{\nu^-}\bigg(\begin{array}{c}-x_2\\x_1\end{array}\bigg),&|\vv x|<R\\[2ex]\dfrac{R^2-|\vv x|^2}{\nu^+}\bigg(\begin{array}{c}-x_2\\x_1\end{array}\bigg),&|\vv x|\ge R\end{cases},\quad\text{and}\quad p(\vv x)=x_2^2-x_1^2,
\end{equation}
We observe that $\jump{\vv u}=\vv 0$ and $\jump{\nu D(\vv u)}=0$, and since $p$ is continuous, we get $\jump{\sigma(\vv u,p)\vv n}=\vv 0$. Also, $\vv u$ is divergence free. We will show that the errors $\|\nu\,D(\vv u-\vv u_h)\|_\Omega$ and $\|p-p_h\|_\Omega$ are independent of $\nu$. In order to do this, we consider a uniform diagonal triangular decomposition of $\Omega$, and test the code with two pairs of spaces $\vv V_h^{\rm bulk}\times M_h^{\rm bulk}$ that satisfy Assumption~\ref{Ass3}. These are the Mini-element and the pair $P_2$ -- $P_0$. We choose a mesh with $N=160092$ degrees of freedom for the Mini-element and $N=164697$ degrees of freedom for the pair $P_2$ -- $P_0$.

We denote the errors by ${\tt e}(\vv u):=\|\nu D(\vv u-\vv u_h)\|_{\Omega}$ and ${\tt e}(p):=\|p-p_h\|_{\Omega}$ and compute them experimentally, with decreasing values of $\nu^-$ and increasing values of $\nu^+$. As a good balance between stability and conditioning, we set the stabilization parameters to $\gamma=25$, $\gamma_{\vv u}^\pm=15$ and $\gamma_p^\pm=20$ for the Mini-element and $\gamma=20$, $\gamma_{\vv u}^\pm=10$ and $\gamma_p^\pm=15$ for the pair $P_2$ -- $P_0$.
The numerical results are summarized in Table~\ref{t:ex1a}.

\begin{table}[!ht]
	\begin{center}
		\begin{tabular}{c|c||c|c||c|c}
			\multicolumn{2}{c||}{Parameters} &
			\multicolumn{2}{|c||}{$P_2$ -- $P_0$} &
			\multicolumn{2}{c}{Mini-element} \\
			\hline
			$\nu^-$ & $\nu^+$ & ${\tt e}(\vv u)$ & ${\tt e}(p)$ & ${\tt e}(\vv u)$ & ${\tt e}(p)$\\
			\hline
			$1E-01$ & $ 1E+01$ & $0.01708$ & $0.00448$ & $0.06200$ & $0.00409$\\
			$1E-02$ & $ 1E+02$ & $0.01708$ & $0.00448$ & $0.06200$ & $0.00409$\\
			$1E-03$ & $ 1E+03$ & $0.01708$ & $0.00448$ & $0.06200$ & $0.00409$\\
			$1E-04$ & $ 1E+04$ & $0.01708$ & $0.00448$ & $0.06200$ & $0.00409$
		\end{tabular}
	\end{center}	
	\caption{Example 1 with $\Omega^-$ completely interior.  Errors are shown for a fixed mesh, increasing values of $\nu^+$ and decreasing values of $\nu^-$. \label{t:ex1a}}	
\end{table}
From Table~\ref{t:ex1a}, we observe that the errors ${\tt e}(\vv u)$ and ${\tt e}(p)$ remain unchanged for a fixed mesh when $\nu^-$ decreases and $\nu^+$ increases.

\medskip

Now we switch the role of $\Omega^+$ and $\Omega^-$, and define $\Omega^{+}=\{\vv x\in\mathbb{R}^2\,:\,|\vv x|<R\}$, $\Omega^{-}=\Omega\setminus\overline{\Omega}^{+}$.  Observe that $\Omega^+$ is now the inclusion. We choose the data $\vv f\in \vv L^2(\Omega)$ so that the exact solution $(\vv u,p)$ is given by
\begin{equation*}
\vv u(\vv x)=\begin{cases}\dfrac{R^2-|\vv x|^2}{\nu^+}\bigg(\begin{array}{c}-x_2\\x_1\end{array}\bigg),&|\vv x|<R\\[2ex]\dfrac{R^2-|\vv x|^2}{\nu^-}\bigg(\begin{array}{c}-x_2\\x_1\end{array}\bigg),&|\vv x|\ge R\end{cases},\quad\text{and}\quad p(\vv x)=x_2^2-x_1^2.
\end{equation*}
We choose a mesh with $N=160092$ degrees of freedom for the Mini-element and $N=164697$ degrees of freedom for the pair $P_2$ -- $P_0$, and as a good balance between stability and conditioning we set the stabilization parameters to $\gamma=25$, $\gamma_{\vv u}^\pm=15$ and $\gamma_{p}^\pm=20$ for the Mini-element, and $\gamma=20$, $\gamma_{\vv u}^\pm=10$ and $\gamma_p^\pm=15$ for the pair $P_2$ -- $P_0$, consider decreasing values of $\nu^-$ and increasing values of $\nu^+$, and summarize the results in Table~\ref{t:ex1b}.

\begin{table}[!ht]
	\begin{center}
		\begin{tabular}{c|c||c|c||c|c}
			\multicolumn{2}{c||}{Parameters} &
			\multicolumn{2}{|c||}{$P_2$ -- $P_0$} &
			\multicolumn{2}{c}{Mini-element} \\
			\hline
			$\nu^-$ & $\nu^+$ & ${\tt e}(\vv u)$ & ${\tt e}(p)$ & ${\tt e}(\vv u)$ & ${\tt e}(p)$\\
			\hline
			$1E-01$ & $ 1E+01$ & $0.01210$ & $0.00282$ & $0.06200$ & $0.00401$\\
			$1E-02$ & $ 1E+02$ & $0.01210$ & $0.00282$ & $0.06200$ & $0.00401$\\
			$1E-03$ & $ 1E+03$ & $0.01210$ & $0.00282$ & $0.06200$ & $0.00401$\\
			$1E-04$ & $ 1E+04$ & $0.01210$ & $0.00282$ & $0.06200$ & $0.00401$
		\end{tabular}
	\end{center}
	\caption{Example 1 with $\Omega^+$ completely interior.  Errors are shown for a fixed mesh, increasing values of $\nu^+$ and decreasing values of $\nu^-$. \label{t:ex1b}}	
\end{table}
Similarly, we observe that the errors ${\tt e}(\vv u)$ and ${\tt e}(p)$ remain unchanged for a fixed mesh when $\nu^-$ decreases and $\nu^+$ increases.
\subsection*{Example 2}
We consider the same exact solution $(\vv u,p)$ given by \eqref{eq:example1}, and the finite element errors
\[
{\tt e}(\vv u):=\|\nu D(\vv u-\vv u_h)\|_\Omega,\quad{\tt e}( p):=\|p-p_h\|_\Omega,\quad\text{and} \quad{\tt e}(\vv u,p)^2=\|\nu D(\vv u-\vv u_h)\|_\Omega^2+\|p-p_h\|_\Omega^2.
\]
We test the method for $P_2$ -- $P_0$ bulk spaces and fixed viscosity and stabilization parameters $\nu^-=0.5$, $\nu^+=20$, $\gamma=20$, $\gamma_{\vv u}^\pm=10$ and $\gamma_p^\pm=15$. We consider a sequence of uniform triangular meshes with decreasing mesh size.
The experimental rates of convergence are computed as
\[
	\frac{\log(\Phi_j/\Phi_{j-1})}{\log(h_j/h_{j-1})},
\]
where $\Phi_j$ is the corresponding error norm at mesh level $j$.
The error norms and experimental rates are shown in Table~\ref{t:ex2a}. Also, we show plots of the approximate solution $(\vv u_h,p_h)$
in Figure~\ref{fig:ex2}.
\begin{table}[!ht]
	\begin{center}
		\begin{tabular}{c|c|c|c|c|c|c}
			dofs & ${\tt e}(\vv u)$ & ${\tt r}(\vv u)$ & ${\tt e}(p)$ & ${\tt r}(p)$ & ${\tt e}(\vv u,p)$ & ${\tt r}(\vv u,p)$\\
			\hline
			$815$ & $5.9E-01$ & $-$ & $9.8E-01$ & $-$ & $1.1E+00$ & $-$ \\\hline
			$2867$ & $2.6E-01$ & $ 1.201$ & $2.2E-01$ & $ 2.157$ & $3.4E-01$ & $ 1.761$ \\\hline
			$10867$ & $1.1E-01$ & $ 1.163$ & $1.0E-01$ & $ 1.106$ & $1.5E-01$ & $ 1.138$ \\\hline
			$42199$ & $5.2E-02$ & $ 1.139$ & $4.6E-02$ & $ 1.154$ & $6.9E-02$ & $ 1.145$ \\\hline
			$166303$ & $2.5E-02$ & $ 1.039$ & $2.2E-02$ & $ 1.035$ & $3.4E-02$ & $ 1.037$ \\\hline
			$660243$ & $1.3E-02$ & $ 1.010$ & $1.1E-02$ & $ 0.998$ & $1.7E-02$ & $ 1.005$ \\\hline
			$1481863$ & $8.4E-03$ & $ 1.004$ & $7.5E-03$ & $ 0.996$ & $1.1E-02$ & $ 1.000$
		\end{tabular}
	\end{center}
	\caption{Example 2, errors for a sequence of uniform meshes, and fixed values of $\nu^\pm$. The solution $(\vv u_h,p_h)$ is approximated with $P_2$ -- $P_0$ elements.\label{t:ex2a}}	
\end{table}

\begin{figure}[!ht]
	\begin{center}
		\includegraphics[scale=.53]{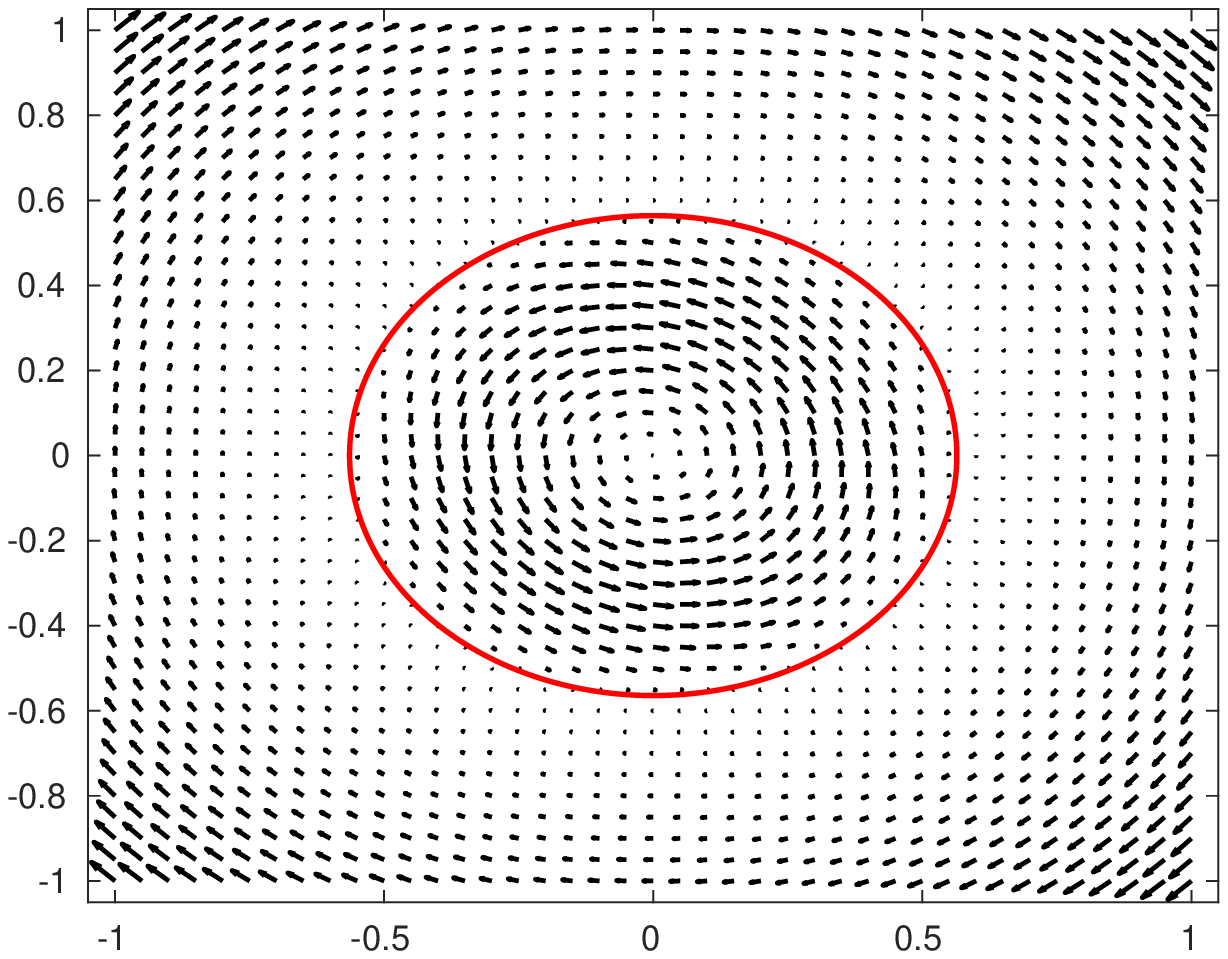}
		\includegraphics[scale=.53]{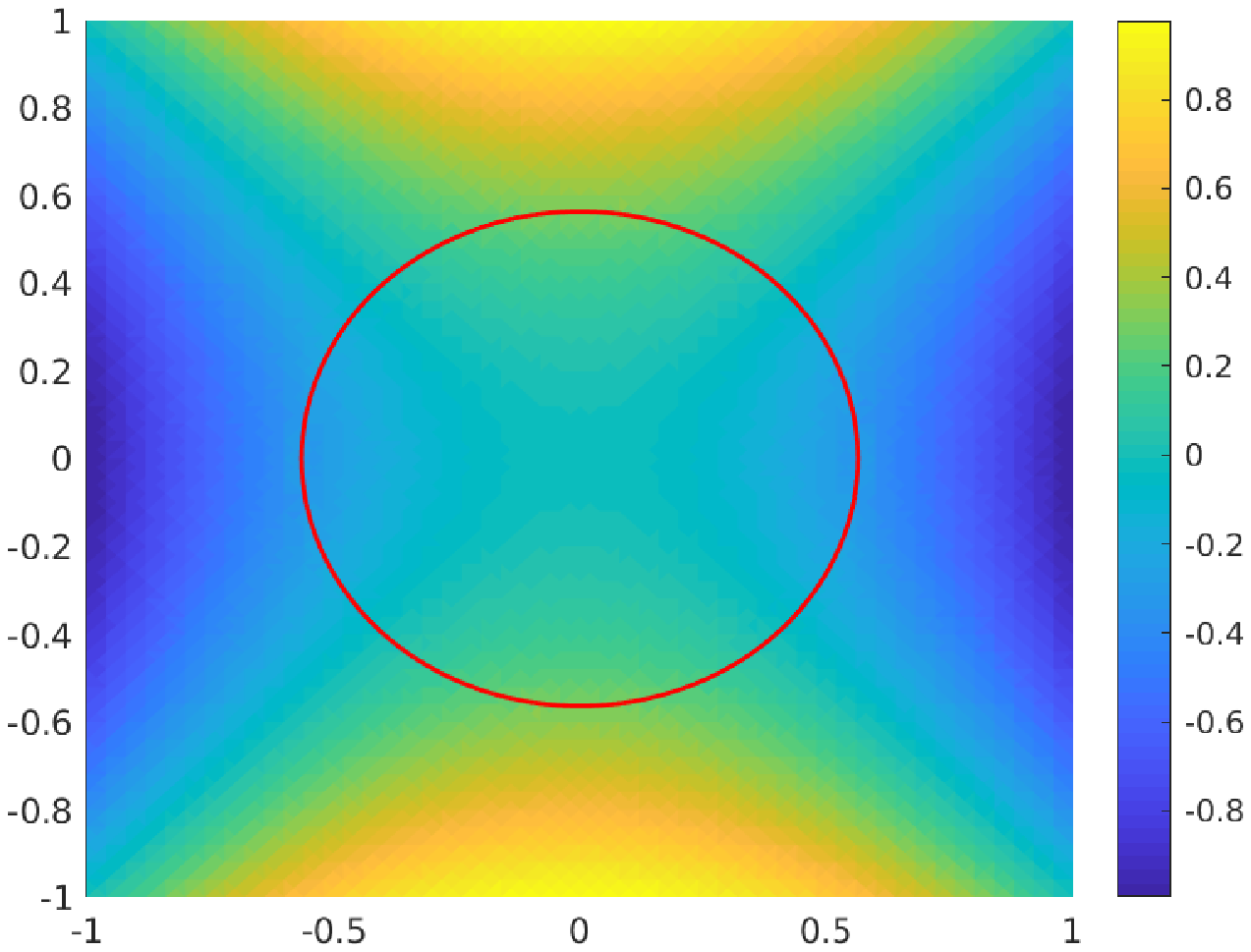}
		\caption{$\vv u_h$ (left) and $p_h$ (right) using the pair $P_2$ -- $P_0$. \label{fig:ex2}}
	\end{center}
\end{figure}

We repeat the experiment using the Mini-element and the same set of parameters and meshes. The computed error norms and experimental rates are shown in Table~\ref{t:ex2b}.

\begin{table}[!ht]
	\begin{center}
		\begin{tabular}{c|c|c|c|c|c|c}
			dofs & ${\tt e}(\vv u)$ & ${\tt r}(\vv u)$ & ${\tt e}(p)$ & ${\tt r}(p)$ & ${\tt e}(\vv u,p)$ & ${\tt r}(\vv u,p)$\\
			\hline
			$604$ & $1.8E+00$ & $-$ & $9.5E-01$ & $-$ & $2.0E+00$ & $-$ \\\hline
			$2072$ & $7.5E-01$ & $ 1.226$ & $2.0E-01$ & $ 2.269$ & $7.8E-01$ & $ 1.362$ \\\hline
			$7736$ & $3.3E-01$ & $ 1.189$ & $6.0E-02$ & $ 1.716$ & $3.4E-01$ & $ 1.213$ \\\hline
			$29796$ & $1.5E-01$ & $ 1.121$ & $2.0E-02$ & $ 1.587$ & $1.5E-01$ & $ 1.132$ \\\hline
			$116924$ & $7.4E-02$ & $ 1.043$ & $7.0E-03$ & $ 1.507$ & $7.4E-02$ & $ 1.049$ \\\hline
			$463192$ & $3.7E-02$ & $ 1.012$ & $2.5E-03$ & $ 1.470$ & $3.7E-02$ & $ 1.015$ \\\hline
			$1038836$ & $2.4E-02$ & $ 1.005$ & $1.4E-03$ & $ 1.460$ & $2.4E-02$ & $ 1.006$
		\end{tabular}
	\end{center}
	\caption{Example 2, errors for a sequence of uniform meshes, and fixed values of $\nu^\pm$. The solution $(\vv u_h,p_h)$ is approximated with the Mini-element. \label{t:ex2b}}	
\end{table}
\subsection*{Example 3} We consider the parameters $\nu^-=0.5$, $\nu^+=2$, and the exact solution $(\vv u,p)$ given by
\begin{equation}\label{example3}
\vv u(\vv x)=\frac{1}{\pi}\Bigg(\begin{array}{c}\sin\pi x_1\sin\pi x_2\\[1ex]\cos\pi x_1\cos\pi x_2\end{array}\Bigg),\quad\text{and}\quad p(\vv x)=\begin{cases}x_1^2+x_2^2,&|\vv x|<R,\\[2ex]\dfrac{-1}{6\pi},&|\vv x|> R\end{cases}.
\end{equation}
We observe that $\jump{\vv u}=\vv 0$, and the jump $\jump{\sigma(\vv u,p)\vv n}$ is non-zero and is given by
\[\jump{\sigma(\vv u,p)\vv n}=\frac{-3\cos\pi x_1\sin \pi x_2}{2R}\binom{x_1}{-x_2}-\dfrac{6\pi R^2+1}{6\pi R}\binom{x_1}{x_2}=:\bm{\lambda}(\vv x).\]Additionally, $\vv u$ is divergence-free and $(p,1)_\Omega=0$. We test the code with the stabilization parameters $\gamma=30$, $\gamma_{\vv u}^\pm=25$ and $\gamma_p^\pm=25$ for the Mini-element, $\gamma=30$, $\gamma_{\vv u}^\pm=25$ and $\gamma_{p}^\pm=20$ for the pair $P_2$ -- $P_0$, decompose the domain by a sequence of uniform meshes with decreasing size, and summarize the results in the Tables~\ref{T:ex3a} and~\ref{T:ex3b}.

\begin{table}[!ht]
	\begin{center}
		\begin{tabular}{c|c|c|c|c|c|c}
			dofs & ${\tt e}(\vv u)$ & ${\tt r}(\vv u)$ & ${\tt e}(p)$ & ${\tt r}(p)$ & ${\tt e}(\vv u,p)$ & ${\tt r}(\vv u,p)$\\
			\hline
			$815$ & $1.4E+00$ & $-$ & $5.8E+00$ & $-$ & $5.9E+00$ & $-$ \\\hline
			$2867$ & $4.6E-01$ & $ 1.568$ & $2.9E-01$ & $ 4.321$ & $5.4E-01$ & $ 3.445$ \\\hline
			$10867$ & $1.6E-01$ & $ 1.484$ & $8.1E-02$ & $ 1.824$ & $1.8E-01$ & $ 1.564$ \\\hline
			$42199$ & $7.2E-02$ & $ 1.192$ & $3.3E-02$ & $ 1.309$ & $7.9E-02$ & $ 1.213$ \\\hline
			$166303$ & $3.3E-02$ & $ 1.130$ & $1.3E-02$ & $ 1.359$ & $3.5E-02$ & $ 1.164$ \\\hline
			$660243$ & $1.6E-02$ & $ 1.033$ & $6.0E-03$ & $ 1.104$ & $1.7E-02$ & $ 1.042$ \\\hline
			$1481863$ & $1.1E-02$ & $ 0.999$ & $4.0E-03$ & $ 0.962$ & $1.1E-02$ & $ 0.995$ \\
		\end{tabular}
	\end{center}
	\caption{Example 3, errors for a sequence of uniform meshes, and fixed values of $\nu^\pm$. The solution $(\vv u_h,p_h)$ is approximated with $P_2$ -- $P_0$ elements. \label{T:ex3a}}	
\end{table}

\begin{table}[!ht]
	\begin{center}
		\begin{tabular}{c|c|c|c|c|c|c}
			dofs & ${\tt e}(\vv u)$ & ${\tt r}(\vv u)$ & ${\tt e}(p)$ & ${\tt r}(p)$ & ${\tt e}(\vv u,p)$ & ${\tt r}(\vv u,p)$\\
			\hline
			$604$ & $3.0E+00$ & $-$ & $1.1E+01$ & $-$ & $1.1E+01$ & $-$ \\\hline
			$2072$ & $1.5E+00$ & $ 1.003$ & $1.7E+00$ & $ 2.639$ & $2.3E+00$ & $ 2.288$ \\\hline
			$7736$ & $5.9E-01$ & $ 1.329$ & $4.5E-01$ & $ 1.927$ & $7.4E-01$ & $ 1.609$ \\\hline
			$29796$ & $2.4E-01$ & $ 1.285$ & $1.4E-01$ & $ 1.717$ & $2.8E-01$ & $ 1.415$ \\\hline
			$116924$ & $1.1E-01$ & $ 1.159$ & $4.1E-02$ & $ 1.732$ & $1.2E-01$ & $ 1.261$ \\\hline
			$463192$ & $5.2E-02$ & $ 1.064$ & $1.0E-02$ & $ 2.014$ & $5.3E-02$ & $ 1.134$ \\\hline
			$1038836$ & $3.4E-02$ & $ 1.021$ & $4.8E-03$ & $ 1.850$ & $3.4E-02$ & $ 1.043$ \\
		\end{tabular}
	\end{center}
	\caption{Example 3, errors for a sequence of uniform meshes, and fixed values of $\nu^\pm$. The solution $(\vv u_h,p_h)$ is approximated with the Mini-element.\label{T:ex3b} }	
\end{table}

We finish by showing plots of the computed finite element solution $(\vv u_h,p_h)$ in Figure~\ref{fig:ex3}.
\begin{figure}[!ht]
	\begin{center}
		\includegraphics[scale=.53]{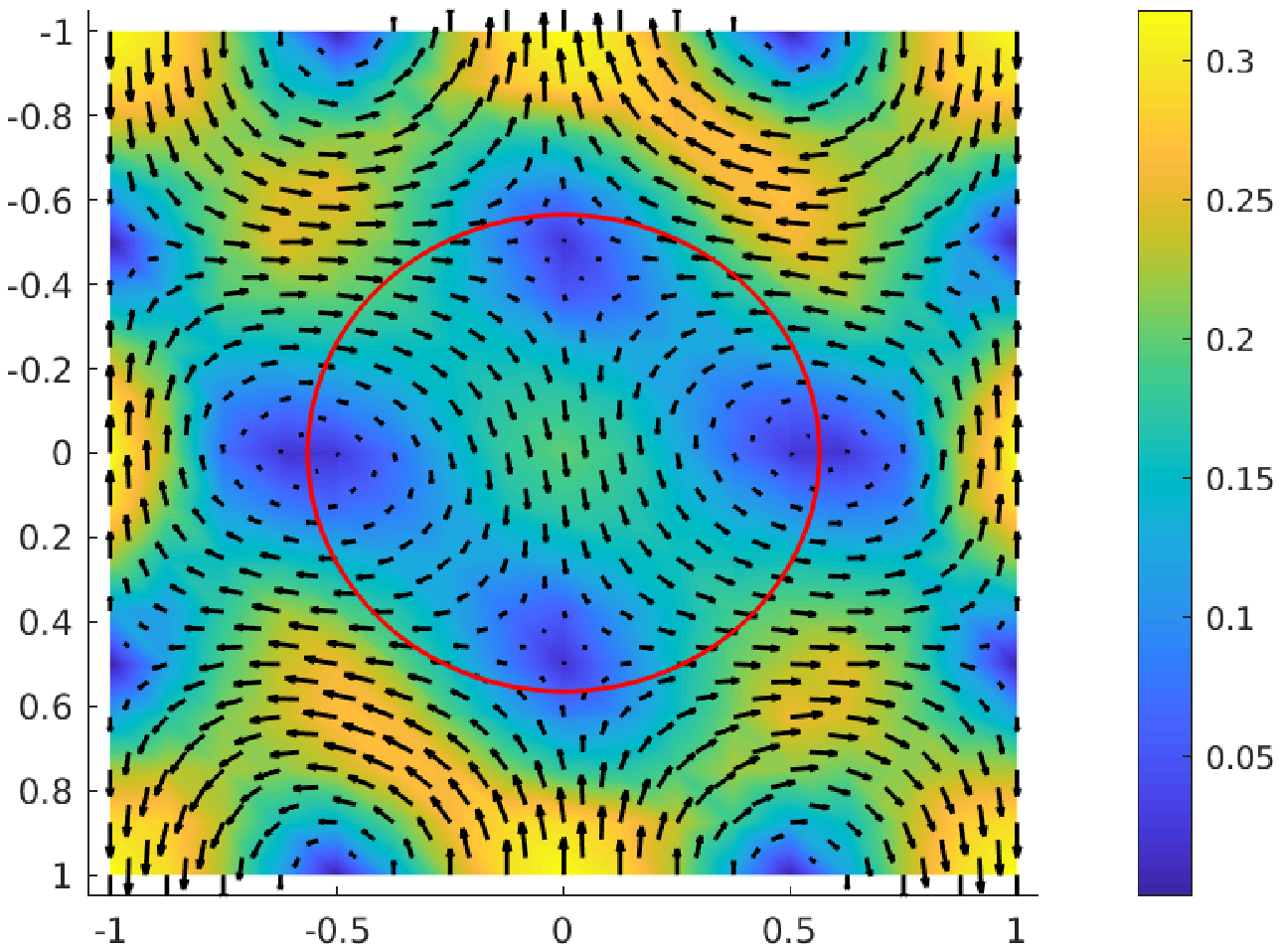}
		\includegraphics[scale=.53]{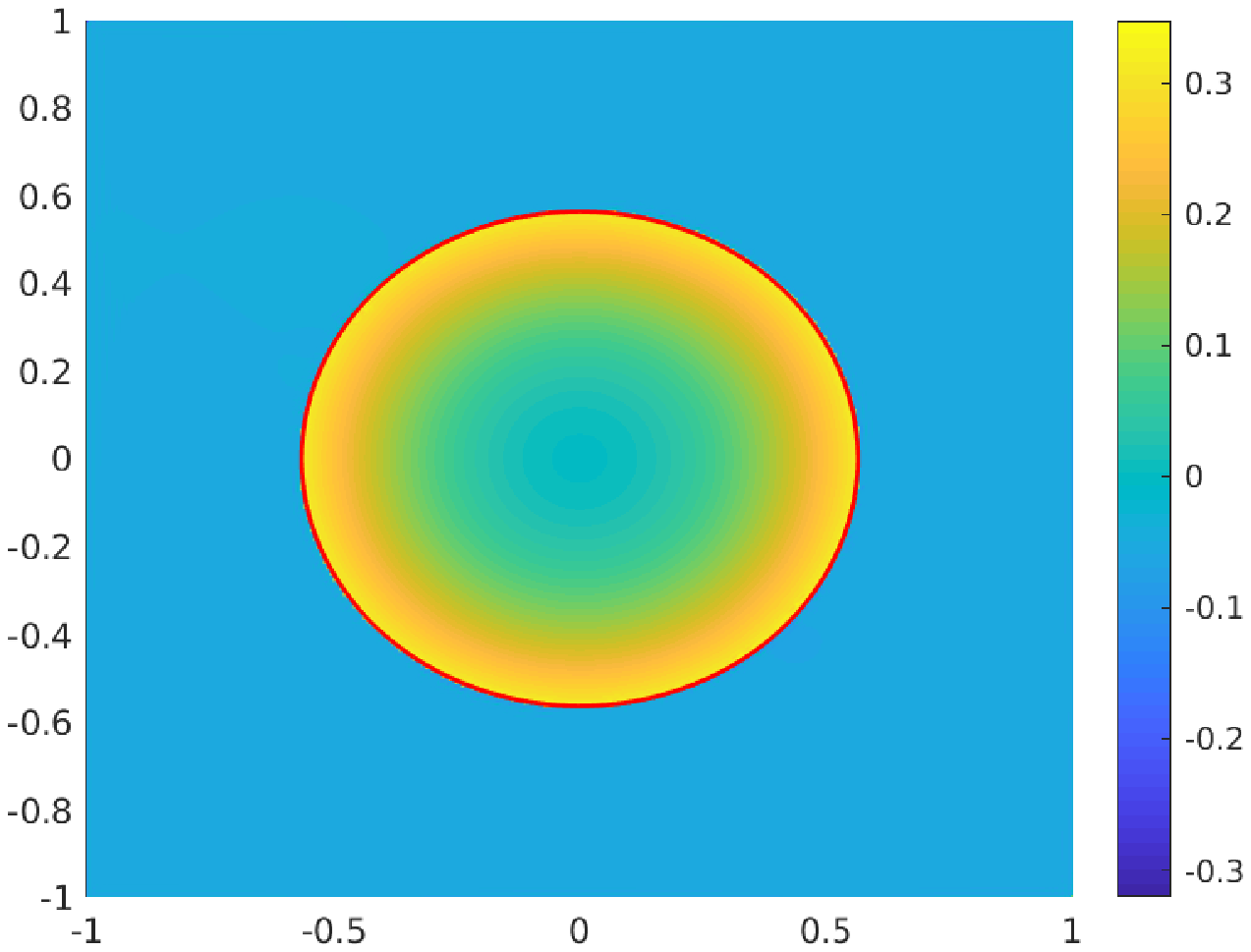}
		\caption{$\vv u_h$ (left) and $p_h$ (right) using the Mini-element. The colormap for the velocity vector field $\vv u_h$ is $\|\vv u_h\|_2$ and for the pressure field $p_h$ is the height $p_h(\vv x)$. \label{fig:ex3}}
	\end{center}
\end{figure}

\subsection*{Example 4} We consider the same exact solution $(\vv u,p)$ given by \eqref{example3} and repeat experiments from Example 3, this time with experimental errors for $\|\div(\vv u-\vv u_h)\|_{\Omega}$, and also with respect to the $L^{\infty}$  norm for the (viscous part of the) stress tensor and pressure
\[
	{\tt e}_d(\vv u):=\|\div(\vv u-\vv u_h)\|_{\Omega},\quad\quad{\tt e}_\infty(\vv u):=\|\nu D(\vv u-\vv u_h)\|_{L^\infty(\Omega)}\quad\text{and} \quad{\tt e}_\infty( p):=\|p-p_h\|_{L^\infty(\Omega)}
\]
respectively. We test the method for $P_2$ -- $P_0$ bulk spaces and fixed viscosity and stabilization parameters $\nu^-=0.5$, $\nu^+=20$, $\gamma=20$, $\gamma_{\vv u}^\pm=10$ and $\gamma_p^\pm=15$. We build a sequence of uniform triangular meshes with decreasing mesh size.  The error norms and experimental rates are shown in Table~\ref{T:ex4a}. 
\begin{table}[!ht]
	\begin{center}
		\begin{tabular}{c|c|c|c|c|c|c}
			dofs & ${\tt e}_{d}(\vv u)$ & ${\tt r}_d(\vv u)$ & ${\tt e}_\infty(\vv u)$ & ${\tt r}_\infty(\vv u)$ & ${\tt e}_\infty(p)$ & ${\tt r}_\infty(p)$\\
			\hline
			$1553$ & $6.8E-02$ & $-$ & $3.7E-01$ & $-$ & $2.3E-01$ & $-$ \\\hline
			$5609$ & $3.0E-02$ & $ 1.183$ & $1.7E-01$ & $ 1.130$ & $1.1E-01$ & $ 1.045$ \\\hline
			$21457$ & $1.7E-02$ & $ 0.847$ & $1.1E-01$ & $ 0.555$ & $6.4E-02$ & $ 0.828$ \\\hline
			$83873$ & $7.3E-03$ & $ 1.184$ & $4.7E-02$ & $ 1.285$ & $3.0E-02$ & $ 1.079$ \\\hline
			$331585$ & $3.6E-03$ & $ 1.024$ & $2.1E-02$ & $ 1.191$ & $1.5E-02$ & $ 1.026$ \\\hline
			$1318473$ & $1.7E-03$ & $ 1.053$ & $9.8E-03$ & $ 1.066$ & $7.4E-03$ & $ 0.991$ \\
		\end{tabular}
	\end{center}
\caption{Example 4, errors for a sequence of uniform meshes, and fixed values of $\nu^\pm$. The solution $(\vv u_h,p_h)$ is approximated with $P_2$ -- $P_0$ elements.\label{T:ex4a}}
\end{table}\\
We repeat the experiment using the Mini-element, the same viscosities $\nu^\pm$ and stabilization parameters $\gamma=30$, $\gamma_{\vv u}^\pm=15$ and $\gamma_p^\pm=5$. The computed error norms and experimental rates are shown in Table~\ref{T:ex4b}. Although, the error estimates in $L^{\infty}$ norm are not covered by the analysis of the paper, the numerical experiments demonstrate the first order of convergence.
\begin{table}[!ht]
	\begin{center}
		\begin{tabular}{c|c|c|c|c|c|c}
			dofs & ${\tt e}_{d}(\vv u)$ & ${\tt r}_d(\vv u)$ & ${\tt e}_\infty(\vv u)$ & ${\tt r}_\infty(\vv u)$ & ${\tt e}_\infty(p)$ & ${\tt r}_\infty(p)$\\
			\hline
			$1122$ & $5.0E-01$ & $-$ & $1.7E+00$ & $-$ & $8.9E-01$ & $-$ \\\hline
			$3994$ & $2.3E-01$ & $ 1.124$ & $9.4E-01$ & $ 0.818$ & $3.7E-01$ & $ 1.277$ \\\hline
			$15154$ & $1.1E-01$ & $ 1.055$ & $4.8E-01$ & $ 0.984$ & $1.4E-01$ & $ 1.432$ \\\hline
			$58978$ & $5.4E-02$ & $ 1.025$ & $2.4E-01$ & $ 0.995$ & $7.2E-02$ & $ 0.916$ \\\hline
			$232642$ & $2.7E-02$ & $ 1.013$ & $1.2E-01$ & $ 0.999$ & $3.8E-02$ & $ 0.905$ \\\hline
			$923994$ & $1.3E-02$ & $ 1.006$ & $6.0E-02$ & $ 1.000$ & $2.0E-02$ & $ 0.964$ \\\hline
			$2074098$ & $8.9E-03$ & $ 1.003$ & $4.0E-02$ & $ 1.000$ & $1.3E-02$ & $ 0.975$ \\
		\end{tabular}
	\end{center}
	\caption{Example 4, errors for a sequence of uniform meshes, and fixed values of $\nu^\pm$. The solution $(\vv u_h,p_h)$ is approximated with the Mini-element.\label{T:ex4b}}
\end{table}

\bibliographystyle{siam} %acm %siam %plain%
\bibliography{references}

\appendix

\section{Finite element extension.}\label{Apdx}
The goal of this section is to provide the $H^1$-bounded extension operator for finite elements of arbitrary degree, i.e. to prove Lemma~\ref{LemmaExt}.  We build the desired extension for each velocity component independently.  Let $V_h^k$ be the $H^1_0(\Omega)$-conformal FE space of degree $k$.  Then $\mathbf{V}_h^{\rm bulk}=\otimes_{j=1}^d V_{j,h}$, and each $V_{j,h}$  satisfies  $V_h^1 \subset V_{j,h} \subset V_h^r$ for some integer $r>0$.   Fix arbitrary $j\in\{1,\dots,d\}$. Using the definition of the local nodal basis, mapping $T\in \mathcal{T}_h$ to the reference simplex and the equivalence of norms in a space of finite dimension, one shows
\begin{equation}\label{aux1230}
\|v\|_{L^\infty(T)} \le C \max_{ \vv y \in \mathcal{N}(T)} |v(\vv y)| \quad \forall~  v \in V_{j,h},\, T \in \mathcal{T}_h,
\end{equation}
with some $C$ independent of $v$ and $T$.

The strategy will be to build an extension operator for piecewise linears first. Then to use that extension operator to build a general extension operator.

\subsection{Extension operator for piecewise linears} \label{A1}
Let $V_h^{1,+}=\{ v|_{\Omega_h^+}: v \in V_{h}^1\}$.  Consider $v \in V_h^{1,+}$ and let $Ev \in H_0^1(\Omega)$ be the Stein extension \cite{SteinBook} such that
\begin{equation}\label{stein}
Ev=v \quad \text{on }~ \Omega^+ \quad  \text{ and }\quad \|Ev\|_{1,\Omega} \le C \|v\|_{1,\Omega^+}.
\end{equation}
We need $I^{\rm SZ}(Ev)$, the Scott-Zhang interpolant  of $Ev$ onto $V_h^1$. The construction of $I^{\rm SZ}(Ev)\in V_h^1$ follows the standard procedure from \cite{ScottZhang}. However, some care is required to ensure that we are recovering the same $P^1$ function in all interior tetrahedra of $\Omega^+$,
\begin{equation}\label{349}
I^{\text{SZ}} Ev|_T=Ev|_T=v|_T  \quad \text{ for all } T \in \mathcal{T}_{h,i}^+.
\end{equation}
To provide  \eqref{349}, we exploit a freedom in choosing the Scott-Zhang interpolant pointed out in \cite{ScottZhang}:
For every vertex $\vv y$ of $\mathcal{T}_h$ we need to associate either a $d$-dimensional simplex to $\vv y$ or a $d-1$-dimensional simplex. If $\vv y$ is a vertex for some $T \in \mathcal{T}_{h,i}^+$ and $\vv y \notin \partial \Omega$, then we associate one of these simplices from $\mathcal{T}_{h,i}^+$ with $\vv y$. By the stability property of the Scott-Zhang interpolant we have
\begin{equation}\label{scottzhang}
\| I^{\text{SZ}} Ev\|_{1,\Omega} \le C \| Ev\|_{1,\Omega}.
\end{equation}

Now we define a discrete extension operator for piecewise linear $v \in V_h^{1,+}$:
\[
E_h^1 v (\vv y)=
\begin{cases}
v(\vv y) &  \text{ if  } \vv y  \text{ is a vertex} ~\text{and}~ \vv y  \in \overline{\Omega_h^+}, \\
 I^{\text{SZ}} Ev(\vv y) & \text{ if }  \vv y   \text{ is a vertex}  ~\text{and}~ \vv y \notin\overline{\Omega_h^+}.
\end{cases}
\]
Note that
\begin{equation*}\label{P1ext}
E_h^1 v=v \qquad  \text{ on } \Omega_h^+.
\end{equation*}
We decompose $\mathcal{T}_{h,i}^-= \mathcal{T}_{h,i}^{-,1} \cup \mathcal{T}_{h,i}^{-, \text{int}}$ where
$\mathcal{T}_{h,i}^{-,1} =\{ T \in \mathcal{T}_{h,i}^-: T \text{ has a vertex } \vv y \text{ such that } \vv y \in K, \text{ for some } K \in \mathcal{T}_h^\Gamma\}$ and  $\mathcal{T}_{h,i}^{-, \text{int}}= \mathcal{T}_{h,i}^{-}\backslash \mathcal{T}_{h,i}^{-, 1}$.

We then see that
\begin{alignat*}{1}
\|\nabla E_h^1 v\|_{\Omega_{h,i}^-}^2 &= \sum_{T  \in \mathcal{T}_{h,i}^-} \|\nabla E_h^1 v\|_{T}^2 \\
&= \sum_{T  \in \mathcal{T}_{h,i}^{-,1}} \|\nabla E_h^1 v\|_{T}^2+  \sum_{T  \in \mathcal{T}_{h,i}^{-,\text{int}}} \|\nabla E_h^1 v\|_{T}^2 \\
&= \sum_{T  \in \mathcal{T}_{h,i}^{-,1}} \|\nabla E_h^1 v\|_{T}^2+  \sum_{T  \in \mathcal{T}_{h,i}^{-,\text{int}}} \|\nabla I^{\text{SZ}} E v\|_{T}^2 \\
&\le \sum_{T  \in \mathcal{T}_{h,i}^{-,1}} \|\nabla E_h^1 v\|_{T}^2 + \| I^{\text{SZ}} Ev\|_{1,\Omega^-}^2.
\end{alignat*}
Therefore, we are left to bound   $\sum_{T  \in \mathcal{T}_{h,i}^{-,1}} \|\nabla E_h^1 v\|_{T}^2$. We use the triangle inequality to get
\begin{alignat*}{1}
\sum_{T  \in \mathcal{T}_{h,i}^{-,1}} \|\nabla E_h^1 v\|_{T}^2 & \le \sum_{T  \in \mathcal{T}_{h,i}^{-,1}} 2( \|\nabla I^{\text{SZ}} Ev\|_{T}^2+\|\nabla (E_h^1 v-I^{\text{SZ}} Ev)\|_{T}^2) \\
& \le 2 \| I^{\text{SZ}} Ev\|_{1,\Omega^-}^2+ 2 \sum_{T  \in \mathcal{T}_{h,i}^{-,1}} \|\nabla (E_h^1 v-I^{\text{SZ}} Ev)\|_{T}^2.
\end{alignat*}
For ease of notation we set $w= E_h^1 v-I^{\text{SZ}} Ev$. Now note that  if $T \in T  \in \mathcal{T}_{h,i}^{-,1}$ then $w$ vanishes on all vertices that do not belong to $\overline{\Omega_h^+}$. Hence, using inverse estimates we get
\begin{alignat*}{1}
 \|\nabla w \|_{T}^2 \le C \, h_T^{-2} \| w\|_{T}^2 & \le  C h_T^{d-2} \| w\|_{L^\infty(T)}^2 \le C h_T^{d-2} \sum_{K \in \mathcal{T}_h^\Gamma, \overline{K} \cap \overline{T} \neq \emptyset }  \| w\|_{L^\infty(K)}^2 \\
 & \le C h_T^{-2} \sum_{K \in \mathcal{T}_h^\Gamma, \overline{K} \cap \overline{T} \neq \emptyset }  \|w\|_{K}^2.
\end{alignat*}
Hence, we will have
\begin{equation*}
\sum_{T  \in \mathcal{T}_{h,i}^{-,1}} \|\nabla w\|_{T}^2 \le C \sum_{T \in \mathcal{T}_h^\Gamma} h_T^{-2} \| w\|_{T}^2.
\end{equation*}

Recalling Assumption~\ref{Ass1}, if $T \in \mathcal{T}_h^{\Gamma}$ there exists  $K_T \in \mathcal{T}_{h,i}^+$ such that $T=K_1,K_2,\ldots,K_{\ell}=K_T$ and  $K_j,  K_{j+1}$ have a common $d-1$ face for $j <\ell$ and $K_j \subset \Omega_h^+$. The number $\ell \le M$ where $M$ is uniformly bounded  and only depends on the shape regularity of the mesh. Then we see from a simple scaling argument that
\begin{equation*}
\| w\|_{T}^2\le C (\|w\|_{K_T}^2 + \sum_{i=1}^\ell h_{K_i}^2 \|\nabla w\|_{K_i}^2).
\end{equation*}
Using \eqref{349} we have $w|_{K_T} \equiv 0$ since $K_T \in \mathcal{T}_{h,i}^+$ and so
\begin{equation*}
\| w\|_{T}^2\le C \sum_{i=1}^\ell h_{K_i}^2 \|\nabla w\|_{K_i}^2 \le C \,h_T^2 \sum_{i=1}^\ell \|\nabla w\|_{K_i}^2.
\end{equation*}
In the last inequality we used that by shape regularity $h_{K_i} \le C h_T$ where $C$ depends on $M$ and shape regularity constant.
We then get
\begin{equation*}
\sum_{T \in \mathcal{T}_h^\Gamma} h_T^{-2} \| w\|_{T}^2 \le C \sum_{T \in \mathcal{T}_h^+}\| \nabla w\|_{T}^2 \le C (\| \nabla  v\|_{\Omega_h^+}^2+ \| \nabla (I^{\text{SZ}} Ev)\|_{\Omega_h^+}^2).
\end{equation*}
In the last step we used \eqref{P1ext}.  Therefore, combining all the inequalities above we obtain
\begin{alignat*}{1}
\|\nabla E_h^1 v\|_{\Omega_{h,i}^-}^2 \le C (\| I^{\text{SZ}} Ev\|_{1,\Omega}^2+ \| \nabla  v\|_{\Omega_h^+}^2).
\end{alignat*}
We hence, get that after using \eqref{stein} and \eqref{scottzhang} that
\begin{equation*}
\|\nabla E_h^1 v\|_{\Omega_{h,i}^-} \le  C \| v\|_{1,\Omega_h^+}.
\end{equation*}
Finally, since $E_h^1v|_{\Omega_h^+}=v$ and $\overline{\Omega}=\overline{\Omega_{h,i}^-}\cup\overline{\Omega_h^+}$ as indicated in section~\ref{sec:prelimdisc}, we get
\begin{equation}\label{E1estimate}
\|\nabla E_h^1 v\|_{\Omega} \le  C \| v\|_{1,\Omega_h^+}.
\end{equation}

\subsection{General Discrete Extension Operator}
Building on the availability of $E_h^1$ we  define the general extension operator. Let $V_h^+= \{ v|_{\Omega_h^+}: v \in V_{j,h}\}$ and consider the subspace $W_h^+=\{v \in V_h^+: v(\vv y)=0 \text{ for all vertices } \vv y \text{ of } \mathcal{T}_h^+ \}$.  For $v \in W_h^+$ we consider the extension $Q_h v  \in V_{j,h}$  by defining its nodal values as follows
\[
Q_hv (\vv y)=
\begin{cases}
v(\vv y) &  \text{ if  } \vv y \in \mathcal{N}(\mathcal{T}_h)~\text{and}~ \vv y  \in \overline{\Omega_h^+}\ \\
0 & \text{ if } \vv y \in  \mathcal{N}(\mathcal{T}_h)~\text{and}~ \vv y \notin\overline{\Omega_h^+}
\end{cases}
\]
Then we can easily prove the following lemma.
\begin{lemma} For $v \in W_h^+$ it holds
\begin{equation}\label{J}
\|\nabla(Q_h v)\|_{\Omega_{h,i}^-} \le C_1 \| \nabla v\|_{\Omega_h^+},
\end{equation}
where the constant $C_1$ only depends on the shape regularity of the mesh.
\end{lemma}
\begin{proof}
Note that $Q_h v|_{\Omega_{h,i}^-}$ is supported on all simplices $T \in \mathcal{T}_h^-$ that have at least one edge belonging to $\partial \Omega_h^+$. Lets call this set of  simplices $\omega_h^-$:
\begin{equation*}
\omega_h^-=\{ T \in \mathcal{T}_h^-: \text{ an edge of } T \text{ belongs to } \partial \Omega_h^+ \},
\end{equation*}
and we also consider the set
\begin{equation*}
\omega_h^+=\{ T \in \mathcal{T}_h^+: \text{ an edge of } T \text{ belongs to } \partial \Omega_h^+ \}.
\end{equation*}
 For each $T \in \omega_h^-$, let $\Xi(T):=\{ \tau \in \mathcal{T}_h^+: \tau \text{ and } T \text{ share a common edge}\}$.  Then we have due to the finite element inverse estimates:
 \begin{equation*}
 \|\nabla Q_h v\|_{\Omega_{h,i}^-}^2 =\sum_{T \in \omega_h^-} \|\nabla Q_h v\|_{T}^2 \le C \sum_{T \in \omega_h^-} h_T^{d-2} \|Q_h v\|_{L^\infty(T)}^2.
 \end{equation*}
At the same time, with the help of \eqref{aux1230} we have for each $T \in \omega_h^+$,
\begin{equation*}
  \|Q_h v\|_{L^\infty(T)} \le C \max_{\vv y \in \mathcal{N}(T)} \|Q_h v(\vv y)\| \le  C\max_{\tau \in \Xi(T)} \|v\|_{L^\infty(\tau)}.
 \end{equation*}
  For the second inequality we used that $Q_h v$ vanishes on all nodes except the nodal points of $T$ that belong  $\partial T$ and belong to the boundary of  $\partial \Omega_h^+$.  Since $v \in W_h^+$ and so $v$ vanishes  on all the vertices of such $\tau's$, we obtain
 \begin{equation*}
  \|Q_h v\|_{L^\infty(T)} \le C h_T  \max_{\tau \in \Xi(T)} \|\nabla v\|_{L^\infty(\tau)}.
 \end{equation*}
 Finally, applying inverse estimates gives
 \begin{equation*}
  \|Q_h v\|_{L^\infty(T)} \le C h_T h_T^{-d/2} \|\nabla v\|_{\tau}  \quad \text{ for some } \tau \in \Xi(T).
 \end{equation*}
  Hence,
  \begin{equation*}
 \|\nabla Q_h v\|_{\Omega_{h,i}^-}^2 \le C \sum_{\tau \in \omega_h^+} \|\nabla v\|_{\tau}^2 \le C \|\nabla v\|_{\Omega_h^+}^2.
 \end{equation*}
 \end{proof}

 Let $I_h: C(\Omega_h^+) \rightarrow V_h^{1,+}$ be the Lagrange interpolant. In section~\ref{A1} we defined a stable discrete extension operator $ E_h^1:V_h^{1,+}(\Omega_h^+)\to V_h^1\subset V_{j,h}$.

We finally can define our general \emph{discrete extension operator}. For any $v \in V_{j,h}^{+}$, we define $E_h v \in V_{j,h}$ as follows
 \begin{equation}\label{Eh}
 E_h v:= E_h^1(I_h v) + Q_h(v-I_h v).
 \end{equation}
 Note that $v-I_h v \in W_h^+$ so indeed this definition makes sense.

We use the boundedness of $E_h^1$, \eqref{E1estimate}, \eqref{J} and the stability of $I_h$ to obtain:
\begin{alignat*}{1}
\| \nabla E_h v\|_{\Omega_{h,i}^-} \le  & \| \nabla (E_h^1(I_h v)) \|_{\Omega_{h,i}^-}+ \| \nabla( Q_h(v-I_h v))\|_{\Omega_{h,i}^-} \\
\le & C \| I_h v_h\|_{1,\Omega_h^+}+ C \| \nabla(v-I_h v)\|_{\Omega_h^+}
\le   C \|v\|_{1,\Omega_h^+},
\end{alignat*}
with some $C$ independent of $h$ and the position of $\Gamma$ in the background mesh. Hence, we we have proven Lemma \ref{LemmaExt}.
\end{document}